\documentclass[a4paper]{amsart}

\usepackage{amsmath,amsthm,amsfonts,amssymb}
\allowdisplaybreaks[3]

\usepackage[headings]{fullpage}

\usepackage{graphicx}

\newtheorem{theorem}{Theorem}
\newtheorem{corollary}{Corollary} 
\newtheorem{proposition}{Proposition} 
\newtheorem{lemma}{Lemma}
\theoremstyle{definition}
\newtheorem{remark}{Remark}
\newtheorem{example}{Example}

\usepackage[colorlinks]{hyperref}

\newcommand{\eps}{\varepsilon}

\usepackage{xcolor}
\newcounter{margin}

\begin{document}

\title{A class of weighted Delannoy numbers}


\author{Jos\'e Mar\'ia Grau}
\address{Departamento de Matem\'aticas, Universidad de Oviedo, Oviedo, Spain}
\email{grau@uniovi.es}
\author{Antonio M. Oller-Marc\'en}
\address{Centro Universitario de la Defensa de Zaragoza - IUMA, Zaragoza, Spain}
\email{oller@unizar.es}
\author{Juan Luis Varona}
\address{Departamento de Matem\'aticas y Computaci\'on, Universidad de La Rioja, Logro\~no, Spain}
\email{jvarona@unirioja.es}
\thanks{The third author is supported by grant PID2021-124332NB-C22 from MICINN/FEDER.}

\thanks{\textbf{This paper has been published in:} \textit{Filomat} \textbf{36} (2022), no.~17, 5985--6007;
\url{https://doi.org/10.2298/FIL2217985G}}


\begin{abstract} 
The weighted Delannoy numbers are defined by the recurrence relation $f_{m,n}=\alpha f_{m-1,n}+ \beta f_{m,n-1}+ \gamma f_{m-1,n-1}$ if $m n>0 $, with $f_{m,n}=\alpha^m \beta^n$ if $n m=0$.
In this work, we study a generalization of these numbers considering the same recurrence relation but with $f_{m,n}=A^m B^n$ if $n m=0$. More particularly, we focus on the diagonal sequence $f_{n,n}$. 
With some ingenuity, we are able to make use of well-established methods by Pemantle and Wilson, and by Melczer in order to determine its asymptotic behavior in the case $A,B,\alpha,\beta,\gamma\geq 0$. In addition, we also study its P-recursivity with the help of symbolic computation tools. 
\end{abstract}

\maketitle

\section{Introduction}

The \emph{Delannoy number} $D_{m,n}$ is usually defined as the number of paths on $\mathbb Z^2$ going from $(0,0)$ to $(m,n)$ using only steps $(0,1)$, $(1,0)$ and $(1,1)$. Delannoy numbers are named after the French army officer and amateur mathematician Henri Delannoy, who first introduced them in the late 19th century~\cite{Delannoy}.

It is rather straightforward to see that Delannoy numbers are given by the recursion
\[
D_{m,n} =
\begin{cases}
1 , & \text{if $m n=0$,} \\
D_{m-1,n} + D_{m,n-1} + D_{m-1,n-1}, & \text{if $m n>0$.}
\end{cases}
\]
Moreover, the following closed-form formulas for them can also be easily obtained
\[ 
D_{m,n} = \sum_{i=0}^m \binom{n}{i} \binom{n+m-i}{n}
= \sum_{i=0}^m 2^i \binom{n}{i}\binom{m}{i}.
\]

The table below shows the first values for the Delannoy numbers \cite[OEIS A008288]{A008288}. The bold numbers in this table are the so-called \emph{central Delannoy numbers} $\mathfrak{D}_n := D_{n,n}$ \cite[OEIS A001850]{A001850}.
\[ 
\begin{array}{c|ccccccccc}
m,n & 0 & 1 & 2 & 3 & 4 & 5 & 6 & 7 & 8 \\
\hline
0 & \mathbf{1} & 1 & 1 & 1 & 1 & 1 & 1 & 1 & 1 \\
1 & 1 & \mathbf{3} & 5 & 7 & 9 & 11 & 13 & 15 & 17 \\
2 & 1 & 5 & \mathbf{13} & 25 & 41 & 61 & 85 & 113 & 145 \\
3 & 1 & 7 & 25 & \mathbf{63} & 129 & 231 & 377 & 575 & 833 \\
4 & 1 & 9 & 41 & 129 & \mathbf{321} & 681 & 1289 & 2241 & 3649 \\
5 & 1 & 11 & 61 & 231 & 681 &\mathbf{1683} & 3653 & 7183 & 13073 \\
6 & 1 & 13 & 85 & 377 & 1289 & 3653 & \mathbf{8989} & 19825 & 40081 \\
7 & 1 & 15 & 113 & 575 & 2241 & 7183 & 19825 & \mathbf{48639} & 108545 \\
8 & 1 & 17 & 145 & 833 & 3649 & 13073 & 40081 & 108545 & \mathbf{265729} \\
\end{array}
\]

Central Delannoy numbers have been extensively studied. They arise in several different situations: properties of lattices and posets, domino tilings of the Aztec diamond of order $n$ augmented by an additional row of length $2n$ in the middle \cite{SaHol1994}, alignments between DNA sequences \cite{BIO, BIO2}, etc. In \cite{Su2003} up to 29 different interpretations of these numbers are discussed.

The generating function of the central Delannoy numbers, $G(z)=\sum_{n \ge 0} \mathfrak{D}_n z^n$, is the algebraic function
\[ 
G(z) = \frac{1}{\sqrt{1-6z+z^2}}.
\]
This expression for $G(z)$ is obtained using classical techniques for the diagonal of rational generating functions by means of a resultant or a residue
computation. This closed-form then leads, via singularity analysis, to the following asymptotic value \cite{FlSe2009}:
\[
\mathfrak{D}_n 
= \frac{(3+ 2 \sqrt{2})^n}{\sqrt{\pi} \sqrt{3\sqrt{2}-4}}\left( \frac{1}{2\sqrt{n}}
+ \mathcal{O}(n^{-3/2})\right).
\]

Comtet \cite{Com1970} showed that the coefficients of any algebraic generating function satisfy
a linear recurrence. In the case of central Delannoy numbers we have the following:
\[
(n + 2) \mathfrak{D}_{n+2} - (6n + 9)\mathfrak{D}_{n+1} + (n + 1)\mathfrak{D}_{n}=0.
\]
On the other hand, closed-form expressions such as\footnote{Here $\binom{p}{q}=0$ for $q > p \ge 0$ and the double factorial of negative odd integers $-(2k+1)$ is defined by $(-2k-1)!! = (-1)^k/(2k-1)!! = (-2)^k k!/(2k)!$, $k=0,1,\dots$}
\[
\mathfrak{D}_n = \frac{(-1)^n}{6^n} \sum_{i=0}^n (-1)^i 6^{2i}\, \frac{(2i-1)!!}{(2i)!!}\binom{i}{n-i},
\]
and integral representations like
\[ 
\mathfrak{D}_n = \frac{1}{\pi} \int_{3-2\sqrt{2}}^{3+2\sqrt{2}}
\frac{t^{-n-1}\,dt}{\sqrt{(t-3+2\sqrt{2}) (3+2\sqrt{2}-t)}},
\]
are also known for central Delannoy numbers~\cite{QCSG2018}.  

Several generalizations of Delannoy numbers considering restrictions for the paths between $(0,0)$ and $(m,n)$ have already been studied. Among them, we can mention Schr\"oder numbers \cite{Sch1870}, Motzkin numbers \cite{DonSh1977}, Narayana numbers \cite{Nara1955}, etc. Other possible generalizations are related to the so-called Delannoy polynomials \cite{Wan, Dag20}.

In another, and also natural direction, we can mention the so-called \emph{weighted Delannoy numbers}, that are defined as follows. Given $\alpha, \beta, \gamma \in\mathbb{C}$, we consider paths starting at the origin that remain in the first
quadrant and use only the steps $(1,0)$, $(0,1)$ and $(1,1)$ with respective weights $\alpha$, $\beta$ and~$\gamma$. Then, we define the weight of a path as the product of the weights of the
individual steps that comprise it and, for $m,n \ge 0$, we denote by $W_{m,n}$ the sum of
all the weights of paths connecting the origin to the point $(m,n)$. The numbers $W_{m,n}$ are
precisely the weighted Delannoy numbers and they satisfy the recurrence relation\footnote{With the convention $0^0=1$, if required, when defining the initial conditions for $mn=0$.}
\begin{equation}
\label{eq:defsinAB}
W_{m,n}= \begin{cases}
\alpha^m \beta^n, & \text{if $m n=0$,} \\
\alpha W_{m-1,n}+ \beta W_{m,n-1}+ \gamma W_{m-1,n-1}, & \text{if $m n>0$.}
\end{cases}
\end{equation}

This generalization was considered for the first time in 1971 \cite{FrRo1971, HaKl1971} and it admits multifarious interpretations according to the nature of $\alpha$, $\beta$ and $\gamma$. For instance:
\begin{itemize}
\item If $\alpha$, $\beta$ and $\gamma$ are non-negative integers, $W_{m,n}$ can be interpreted as the number of different paths between $(0,0)$ and $(m,n)$ using $\alpha$ kinds of steps $(1,0)$, $\beta$ kinds of steps $(0,1)$ and $\gamma$ kinds of steps $(1,1)$. In Figure~\ref{fig:abgpaths} we provide an example where, in order to distinguish the different kinds of steps in the same direction, we have used continued, dashed and dotted lines.

\begin{figure}[h]
\begin{center}
\includegraphics[width=\textwidth]{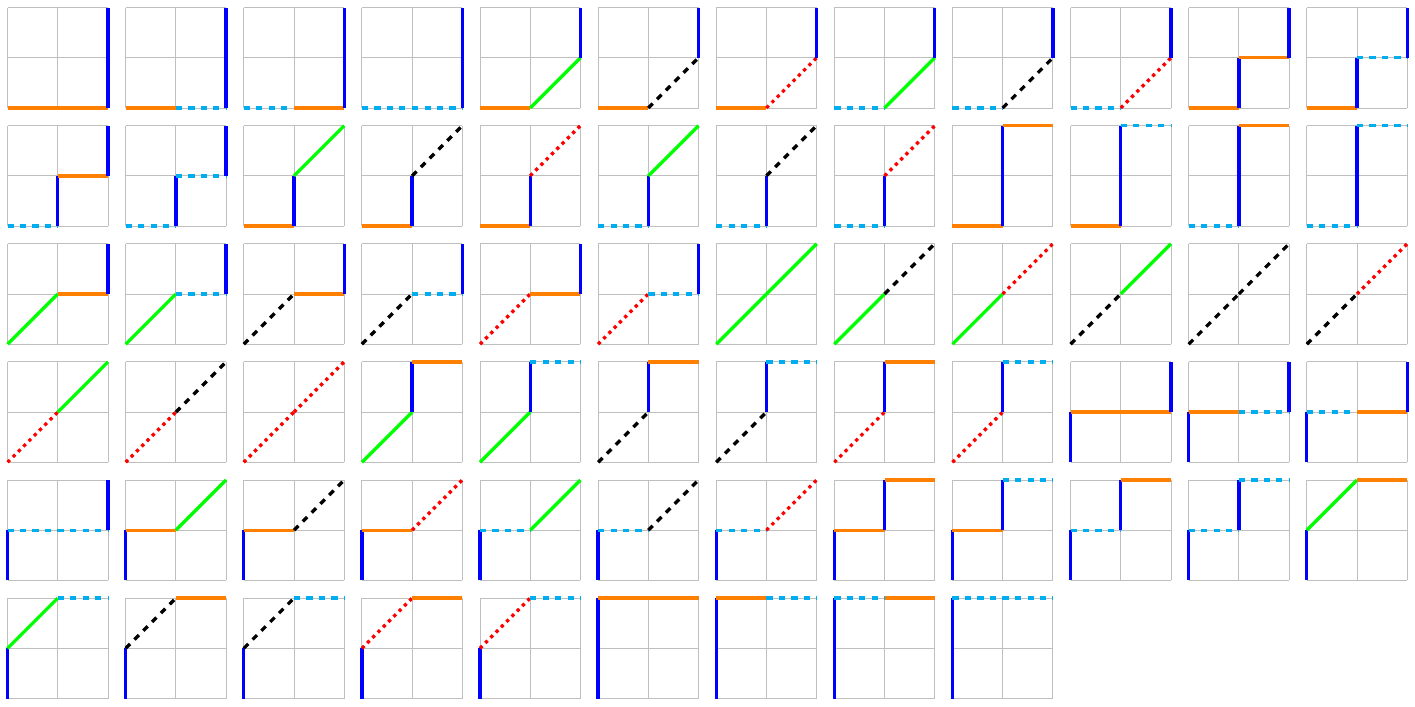}
\caption{For $\alpha=2$, $\beta=1$ and $\gamma=3$  in \eqref{eq:defsinAB}, there are $W_{2,2}=69$ Delannoy paths from $(0,0)$ to $(2,2)$}
\label{fig:abgpaths}
\end{center}
\end{figure}

\item With the same restriction of $\alpha,\beta,\gamma$ being non-negative integer numbers, the following interpretation is also possible. Let us consider that we have letters D of $\gamma$ different colors, letters R of $\alpha$ different colors and letters T of $\beta$ different colors. Then, $W_{m,n}$ represents the number of different words that can be formed in such a way that the number of R's plus that of D's is $m$ and the number of T's plus that of D's is~$n$. In Figure~\ref{fig:abgletras} we provide an example where, besides the color, we have also used different font shapes.

\begin{figure}[h]
\newcommand\tc[2]{\textcolor{#1}{\scalebox{2.2}{#2}}}
\begin{center}
\tc{green}{\textrm{D}}\tc{blue}{\textrm{T}}\quad 
\tc{blue}{\textrm{T}}\tc{green}{\textrm{D}}\quad 
\tc{black}{\textit{D}}\tc{blue}{\textrm{T}}\quad 
\tc{blue}{\textrm{T}}\tc{black}{\textit{D}}\quad 
\tc{red}{\textsf{D}}\tc{blue}{\textrm{T}}\quad 
\tc{blue}{\textrm{T}}\tc{red}{\textsf{D}}
\\[12pt]
\tc{orange}{\textrm{R}}\tc{blue}{\textrm{T}}\tc{blue}{\textrm{T}}\quad
\tc{blue}{\textrm{T}}\tc{orange}{\textrm{R}}\tc{blue}{\textrm{T}}\quad
\tc{blue}{\textrm{T}}\tc{blue}{\textrm{T}}\tc{orange}{\textrm{R}}\quad
\tc{cyan}{\textit{R}}\tc{blue}{\textrm{T}}\tc{blue}{\textrm{T}}\quad
\tc{blue}{\textrm{T}}\tc{cyan}{\textit{R}}\tc{blue}{\textrm{T}}\quad
\tc{blue}{\textrm{T}}\tc{blue}{\textrm{T}}\tc{cyan}{\textit{R}}\quad
\caption{For $\alpha=2$, $\beta=1$ and $\gamma=3$ in \eqref{eq:defsinAB}, we have $W_{1,2}=12$}
\label{fig:abgletras}
\end{center}
\end{figure}

\item If $\alpha, \beta , \gamma \in [0,1]$ and $\alpha+ \beta+ \gamma=1$ then $W_{m,n}$ represents the probability that a random path starting from $(0,0)$ passes through $(m,n)$ assuming that, at a given point $(i,j)$, there are probabilities $\alpha, \beta, \gamma$ of moving to the points $(i+1,j)$, $(i,j+1)$ and $(i+1,j+1)$, respectively. 

\item If $\alpha$, $\beta$ and $\gamma$ are non-negative real numbers, $W_{m,n}$ represents the expected number of paths (under the performance of a certain random variable) between $(0,0)$ and $(m, n)$ provided that $ \alpha, \beta, \gamma$ are the expected number of paths joining $(i, j)$ with $(i+1,j)$, $(i,j+1)$ and $(i+1,j+1)$, respectively.

\item If $\alpha$, $\beta$ and $\gamma$ are any real numbers, $W_{m,n}$ can be interpreted as the amount of matter (if positive) or antimatter (if negative) that will be in the point $(m,n)$ after the process described as follows: 
\begin{enumerate}
\item[(i)] We start with one unit of matter in position $(0,0)$.
\item[(ii)] The amount of matter or antimatter at position $(i,j)$ is multiplied by $\alpha$ and carried to $(i+1, j)$.
\item[(iii)] The amount of matter or antimatter at position $(i,j)$ is multiplied by $\beta$ and carried to $(i, j+1)$.
\item[(iv)] The amount of matter or antimatter at position $(i,j)$ is multiplied by $\gamma$ and carried to $(i+1, j+1)$.
\end{enumerate}
where, of course, any amount of matter is annihilated by any identical amount of antimatter. 
\end{itemize}

Several properties of the weighted Delannoy numbers defined in \eqref{eq:defsinAB} have been established in the literature~\cite{Dag21}. As an example, let us mention the following combinatorial expression~\cite{FrRo1971}:
\[
W_{m,n} = \sum_{k=0}^m \alpha^{m-k} \beta^{n-k} \binom{n}{k} \binom{m}{k} (\alpha\beta+\gamma)^k.
\]

The diagonal sequence $\mathfrak{W}_n:=W_{n,n}$ is of special interest. In \cite{HaKl1971} it is proved that it satisfies the recurrence relation 
\begin{equation} 
\label{eq:recdiagsinAB}
\mathfrak{W}_{n+1} = \frac{(2n +1)(\gamma+ 2 \alpha \beta)}{n+1} 
\mathfrak{W}_n-\frac{\gamma^2 n}{n+1} \mathfrak{W}_{n-1}, 
\quad
\mathfrak{W}_0=1, 
\quad
\mathfrak{W}_1=\gamma+ 2 \alpha \beta.
\end{equation}
Moreover, in \cite{No2012} the asymptotic behavior of $\mathfrak{W}_n$ is investigated,
showing that, for $0<1+\frac{\gamma}{\alpha \beta} \in \mathbb{R}$, one has that\footnote{Here and in what follows, we use the notation $a_n \sim b_n$ with the usual meaning of $a_n/b_n \to 1$ when $n \to \infty$.}
\[
\mathfrak{W}_n \sim \alpha^n \beta^n \,\frac{\left(1+\sqrt{1+\frac{\gamma}{\alpha \beta}}\right)^{2 n+1}}
{2 \,\sqrt[4]{1+\frac{\gamma}{\alpha \beta}}\,\sqrt{\pi n}}.
\]

In this work, we introduce a very natural extension of \eqref{eq:defsinAB} considering the same recurrence relation, but allowing more general initial conditions. Namely, we are interested in the sequence defined by 
\[
f_{m,n} =
\begin{cases}
A^m B^n , & \text{if $m n=0$,} \\
\alpha f_{m-1,n} + \beta f_{m,n-1} + \gamma f_{m-1,n-1}, & \text{if $m n>0$.}
\end{cases}
\]
Note that \eqref{eq:defsinAB} is just the particular case $A=\alpha$ and $B=\beta$.  

Using this generalization, all the interpretations above still hold, the only difference being that the models have different behavior when restricted to the coordinate axes. 
The weight of the steps $(1,0)$ on the horizontal axis is $A$ and the weight of the steps $(0,1)$ on the vertical axis is $B$; the weighting of the diagonal steps is maintained.
To illustrate his, let us compare the example in Figure~\ref{fig:abgpaths} ($A=\alpha=2$, $B=\beta=1$ and $\gamma=3$) with the one in Figure~\ref{fig:11abg}, that has $A=B=1$, $\alpha=2$, $\beta=1$ and $\gamma=3$.

\begin{figure}[h]
\begin{center}
\includegraphics[width=\textwidth]{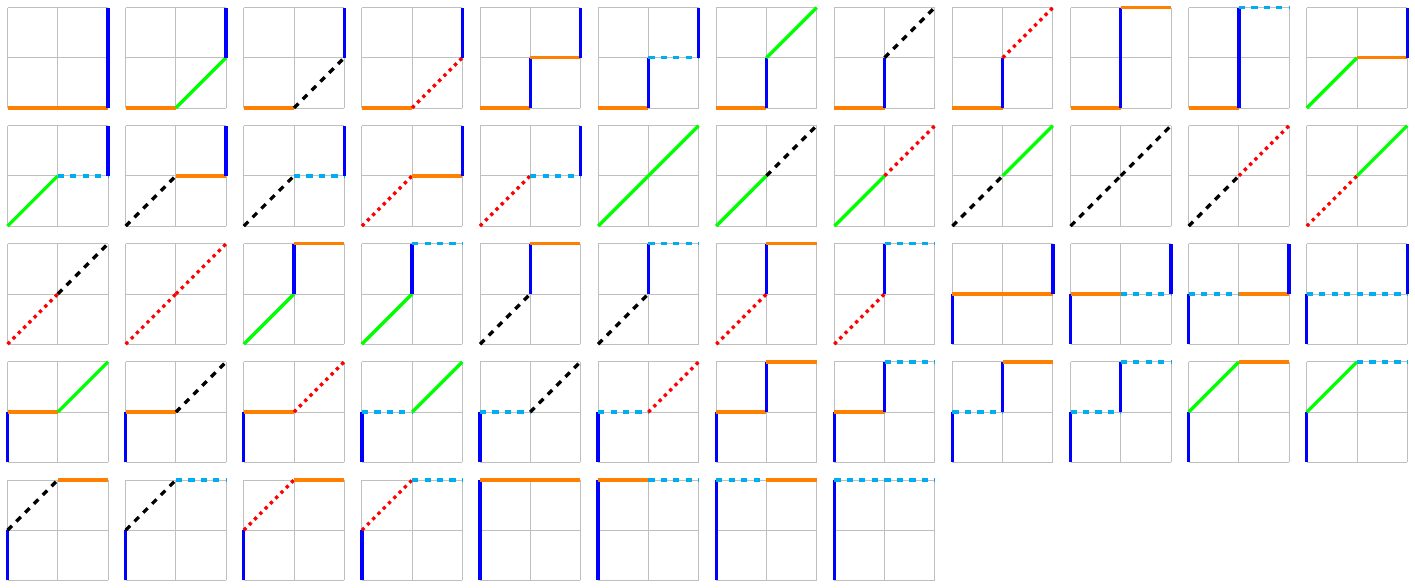}
\caption{For $A=B=1$, $\alpha=2$, $\beta=1$ and $\gamma=3$ in \eqref{eq:recu}, $f_{2,2}=56$}
\label{fig:11abg}
\end{center}
\end{figure}

The paper is organized as follows. In Sections~\ref{sec:GFfmn} and \ref{sec:GFfnn}, we deduce the generating functions of $f_{m,n}$ and of the diagonal sequence $f_{n,n}$.
In Section~\ref{sec:asympfnn}, we study the asymptotic behavior of $f_{n,n}$ and $f_{n+1,n+1}/f_{n,n}$. In Section~\ref{sec:recur}, we show the P-recursive nature of $f_{n,n}$. It is worth mentioning that, in our case, the recurrence relation for the diagonal $\mathfrak{f}_n=f_{n,n}$ can be explicitly stated, but it is much more complicated than in the traditional case \eqref{eq:recdiagsinAB}. Finally, in Section~\ref{sec:future}, we suggest some ideas for further research.

Some of our results are valid for every $A,B,\alpha, \beta, \gamma \in \mathbb{C}$ (the expression for the generating functions, for instance). However, most of the combinatorial interpretations and the results regarding asymptotic behavior require these constants to be, at least, non negative real numbers. In fact, some results are clearly false for negative constants, and a possible generalization for negative values would require a substantial work that does not seem interesting enough. For instance, while Theorem~\ref{th:cocientes} shows that the limit of $f_{n+1,n+1}/f_{n,n}$ always exists for $A,B,\alpha,\beta,\gamma \geq 0$, in Section~\ref{sec:future} we provide several examples in which this limit does not exist for negative values of the parameters.

One final remark. Throughout the paper, we often claim that some computations have been done with the aid of a computer algebra system. We have indistinctly used Maple, Mathematica, Maxima, and SageMath. However, in order to avoid the possible fails of any computer algebra system \cite{DeStWh, DuPeVa}, we have checked all the relevant computations with at least two of them.


\section{Generating function of $f_{m,n}$}
\label{sec:GFfmn}

For $A,B,\alpha, \beta, \gamma \in \mathbb{C}$, let us consider the bivariate sequence $\{f_{m,n}\}_{m,n\geq 0}$ recursively defined by 
\begin{equation}
\label{eq:recu}
f_{m,n} =
\begin{cases}
A^{m}B^{n}, & \text{if } mn=0, \\
\alpha f_{m-1,n} + \beta f_{m,n-1} + \gamma f_{m-1,n-1}, & \text{if } mn>0.
\end{cases}
\end{equation}

Recall that, by definition, the generating function of this sequence is just 
\begin{equation}
\label{eq:fungen}
f(x,y) = \sum_{m=0}^{\infty} \sum_{n=0}^{\infty} f_{m,n} x^m y^n.
\end{equation}
Then, we have the following.

\begin{theorem} 
\label{th:GFfmn}
For every $|x|<1/|A|$ and $|y|<1/|B|$, it holds that
\begin{equation}
\label{eq:funf}
f(x,y) = \frac{1 - \alpha x - \beta y + \alpha B xy + \beta A xy - AB xy}
{(1-Ax)(1-By)(1 - \alpha x - \beta y - \gamma xy)}.
\end{equation}
\end{theorem}

\begin{proof}
Let us identify the function \eqref{eq:fungen},
showing the domain of convergence.

We have
\begin{align*}
f(x,y) &=
1 + \sum_{m=1}^{\infty} A^m x^m + \sum_{n=1}^{\infty} B^n y^n
+ \sum_{m=1}^{\infty} \sum_{n=1}^{\infty} f_{m,n} x^m y^n \\
&= 1 + \frac{Ax}{1-Ax} + \frac{By}{1-By}
+ \alpha \sum_{m=1}^{\infty} \sum_{n=1}^{\infty} f_{m-1,n} x^m y^n \\
&\qquad\qquad
+ \beta \sum_{m=1}^{\infty} \sum_{n=1}^{\infty} f_{m,n-1} x^m y^n
+ \gamma \sum_{m=1}^{\infty} \sum_{n=1}^{\infty} f_{m-1,n-1} x^m y^n,
\end{align*}
where, to sum the geometric progressions, we have used $|Ax|<1$ and $|By|<1$.

Changing $m-1 \mapsto m$ in the first series, $n-1 \mapsto n$ in the second one,
and both in the third one,
\begin{align*}
f(x,y) &= 1 + \frac{Ax}{1-Ax} + \frac{By}{1-By}
+ \alpha x \sum_{m=0}^{\infty} \sum_{n=1}^{\infty} f_{m,n} x^m y^n 
+ \beta y \sum_{m=1}^{\infty} \sum_{n=0}^{\infty} f_{m,n} x^m y^n \\
&\qquad\qquad
+ \gamma xy \sum_{m=0}^{\infty} \sum_{n=0}^{\infty} f_{m,n} x^m y^n \\
&= 1 + \frac{Ax}{1-Ax} + \frac{By}{1-By}
+ \alpha x \bigg( \sum_{m=0}^{\infty} \sum_{n=0}^{\infty} f_{m,n} x^m y^n 
- \sum_{m=0}^{\infty} f_{m,0} x^m \bigg) \\
&\qquad\qquad + \beta y \bigg( \sum_{m=0}^{\infty} \sum_{n=0}^{\infty} f_{m,n} x^m y^n 
- \sum_{n=0}^{\infty} f_{0,n} y^n \bigg)
+ \gamma xy \sum_{m=0}^{\infty} \sum_{n=0}^{\infty} f_{m,n} x^m y^n \\
&= 1 + \frac{Ax}{1-Ax} + \frac{By}{1-By}
+ \alpha x \bigg( f(x,y) - \frac{1}{1-Ax} \bigg) 
+ \beta y \bigg( f(x,y) - \frac{1}{1-By} \bigg)
+ \gamma xy f(x,y) \\
&= 1 + \frac{(A-\alpha)x}{1-Ax} + \frac{(B-\beta)y}{1-By}
+ (\alpha x + \beta y + \gamma xy) f(x,y).
\end{align*}
Then,
\begin{align*}
&(1 - \alpha x - \beta y - \gamma xy) f(x,y)
= 1 + \frac{(A-\alpha)x}{1-Ax} + \frac{(B-\beta)y}{1-By} \\
&\qquad\qquad = \frac{(1-Ax)(1-By) + (A-\alpha)(1-By)x + (B-\beta)(1-Ax)y}{(1-Ax)(1-By)} \\
&\qquad\qquad = \frac{1 - \alpha x - \beta y + \alpha B xy + \beta A xy - AB xy}{(1-Ax)(1-By)}
\end{align*}
and \eqref{eq:funf} is proved.
\end{proof}


\section{Generating function of the diagonal sequence $f_{n,n}$}
\label{sec:GFfnn}

Now we consider the sequence $\{f_{n,n}\}$. By definition, its generating function is just $G(z) = \sum_{n\ge 0} f_{n,n} z^n$. In order to get an explicit expression for $G(z)$, we will use the method described in \cite{HaKl1971, RaWi2007}. It leads to the following.

\begin{theorem}
\label{th:GFfnn}
Let $S := S(z) = \sqrt{1 + \gamma^2 z^2 - 2 (2\alpha\beta+\gamma)z}$. Then, 
\begin{equation}
\label{eq:dossum}
 G(z) = \frac{-B+\beta}{\beta - B + \alpha B^2 z + \gamma B z}
  + \frac{2\alpha z \big(\alpha B + \beta A - AB + \gamma\big) \big({-1} + \gamma z + S\big)}
  {S \big({-1} + 2\alpha B z + \gamma z + S \big) \big(2\alpha + A ({-1} + \gamma z + S)\big)}.
\end{equation}
\end{theorem}

\begin{proof}
Let us start with $f(x,y)$ defined as in \eqref{eq:funf}, which is is rational and holomorphic in a neighborhood of the origin. Then, for a fixed small enough $z$, the function $f(s,z/s)$ will be rational and holomorphic as a function of $s$ in some annulus around $s = 0$. Thus, in that annulus, $f(s,z/s)$ can be represented by a Laurent series whose constant term (coefficient of $s^0$) is $\sum_{m \ge 0} f_{m,m}z^m$, the series we want to compute.

By Cauchy's integral and residue theorems, we have that, for some circle $\Gamma_z$ about $s = 0$,
\begin{equation}
\label{eq:diagonal}
\sum_{m \ge 0} f_{m,m} z^{m}
= f(s,z/s)\big|_{s=0} 
= \frac{1}{2 \pi i} \int_{\Gamma_{z}} \frac{f(s,z/s)}{s} \,ds 
= \sum_{k} \operatorname{Res}\bigg(\frac{f(s,z/s)}{s}; s=s_{k}\bigg),
\end{equation}
where the $s_k$ are the ``small'' singularities of $f(s,z/s)/s$, i.e., the ones satisfying $\lim_{z\to 0} s_k(z) = 0$. Since $f$ is rational, these singularities are poles and algebraic functions of $z$, so that the residue sum, the diagonal generating function, is also an algebraic function of~$z$.

In our case, let us take
\begin{align*}
  f(s,z/s)/s &= \frac{1}{\alpha s^2 - s + \gamma sz + \beta z} 
  + \frac{s(A-\alpha)}{(As-1)(\alpha s^2 - s + \gamma sz + \beta z)} 
  - \frac{z(B-\beta)}{(s-Bz)(\alpha s^2 - s + \gamma sz + \beta z)}
  \\
  &= \frac{-\alpha s^2 + s - AB sz + \beta A sz + \alpha B sz - \beta z}
  {(As-1) (s-Bz) (\alpha s^2 - s + \gamma sz + \beta z)}
\end{align*}
as a function of the complex variable $s$. This function has four poles:
\[
  s_A = 1/A,
  \quad
  s_B = Bz,
  \quad
  s_{\pm} = \frac{1-\gamma z \pm \sqrt{(\gamma z-1)^2 - 4\alpha\beta z}}{2\alpha},
\]
whose corresponding residues are
\[
\operatorname{Res}\bigg(\frac{f(s,z/s)}{s}; s=s_A\bigg)
= \frac{A-\alpha}{\alpha - A + \beta A^2 z + \gamma A z},
\]
\[
\operatorname{Res}\bigg(\frac{f(s,z/s)}{s}; s=s_B\bigg)
= \frac{-B+\beta}{\beta - B + \alpha B^2 z + \gamma B z},
\]
\[
\operatorname{Res}\bigg(\frac{f(s,z/s)}{s}; s=s_{+}\bigg)
= \frac{2\alpha z \big(\alpha B + \beta A - AB + \gamma\big) \big(1-\gamma z + S\big)}
  {S \big(1 - 2\alpha B z - \gamma z + S \big) \big({-2\alpha} + A (1 - \gamma z + S)\big)},
\]
\[
\operatorname{Res}\bigg(\frac{f(s,z/s)}{s}; s=s_{-}\bigg)
= \frac{2\alpha z \big(\alpha B + \beta A - AB + \gamma\big) \big({-1} + \gamma z + S\big)}
  {S \big({-1} + 2\alpha B z + \gamma z + S \big) \big(2\alpha + A ({-1} + \gamma z + S)\big)},
\]
with $S = \sqrt{1 + \gamma^2 z^2 - 2 (2\alpha\beta+\gamma)z}$.

Finally, let us recall that, with the notation of \eqref{eq:diagonal}, we must use only the poles that satisfy $\lim_{z\to 0} s_k(z) = 0$. In our case (see above), only $s_B$ and $s_{-}$ satisfy such condition, and we get \eqref{eq:dossum} as claimed.
\end{proof}


\section{Asymptotic behavior}
\label{sec:asympfnn}

Let $A,B,\alpha, \beta, \gamma\geq 0$ and let us consider the recurrence relation~\eqref{eq:recu}:
\[
f_{m,n} =
\begin{cases}
A^m B^n , & \text{if $m n =0$,} \\
\alpha f_{m-1,n}+ \beta f_{m,n-1}+ \gamma f_{m-1,n-1}, & \text{if $mn>0$.}
\end{cases}
\]

First of all, observe that the case $\alpha=\beta=0$ is trivial because, by a simple induction argument, the
following holds for every $m,n\geq 0$:
\[
f_{m,n}=\gamma^{\min\{m,n\}}A^{m-\min\{m,n\}}B^{n-\min\{m,n\}}.
\]
On the other hand, if we assume that $\alpha\beta\neq 0$ and we define $\widehat{f}_{m,n}= \alpha^{-m} \beta^{-n} f_{m,n}$, $\widehat{A}=\frac{A}{\alpha}$, $\widehat{B}=\frac{B}{\beta}$, and $\widehat{\gamma}=\frac{\gamma}{\alpha \beta}$, it is easy to check that 
\[
\widehat{f}_{m,n} =
\begin{cases}
\widehat{A}^m \widehat{B}^n , & \text{if $m n=0$,} \\
\widehat{f}_{m-1,n} + \widehat{f}_{m,n-1}+ \widehat{\gamma} \widehat{f}_{m-1,n-1}, & \text{if $mn>0$.}
\end{cases}
\]
If only $\alpha=0$, it is enough to define $\widehat{f}_{m,n}:= \beta^{-n} f_{m,n}$, $\widehat{A}=A$, $\widehat{B}=\frac{B}{\beta}$ and $\widehat{\gamma}=\frac{\gamma}{\beta}$ to reach a similar situation.
\[
\widehat{f}_{m,n} =
\begin{cases}
\widehat{A}^m \widehat{B}^n , & \text{if $m n=0$,} \\
\widehat{f}_{m,n-1}+ \widehat{\gamma} \widehat{f}_{m-1,n-1}, & \text{if $mn>0$.}
\end{cases}
\]
Finally, if only $\beta=0$, the same idea applies.

All the previous discussion shows that, without loss of generality, we can assume that $\alpha,\beta\in\{0,1\}$. In this section we will focus on the ``complete'' case $\alpha=\beta=1$. All the ideas and techniques can be easily applied if $\alpha\beta=0$. Thus, in what follows, we will just assume that $\gamma\geq 0$ and consider the sequence $\{f_{m,n}\}_{m,n\geq 0}$ defined by
\begin{equation}
\label{eq:recuab1}
f_{m,n} =
\begin{cases}
A^m B^n , & \text{if $m n = 0$,} \\
f_{m-1,n} + f_{m,n-1}+ \gamma f_{m-1,n-1}, & \text{if $mn>0$.}
\end{cases}
\end{equation}

We begin with an easy result whose proof by induction is straightforward. This proposition characterizes, in particular, the cases for which the diagonal sequence $f_{m,m}$ is a geometric sequence, and it will play an important role later on.

\begin{proposition} 
\label{prop:AmBn}
Let $\{f_{m,n}\}$ be the sequence defined in \eqref{eq:recuab1}. Then, $f_{m,n} = A^m B^n$ for every $m,n\geq 0$ if and only if $AB=A+B+\gamma$.
\end{proposition}

Before we proceed, as an example, let us illustrate the discussion above by stating Proposition~\ref{prop:AmBn} in full generality.

\begin{corollary} 
Let $A,B,\alpha, \beta, \gamma\geq 0$ and $\{f_{m,n}\}$ be the sequence defined in \eqref{eq:recu}. The following hold: 
\begin{itemize}
\item[(i)] If $\alpha\beta\neq 0$, then $f_{m,n} = A^m B^n$ for every $m,n\geq 0$ if and only if $AB=\beta A+\alpha B+\gamma$.
\item[(ii)] If $\alpha=0\neq\beta$, then $f_{m,n} = A^m B^n$ for every $m,n\geq 0$ if and only if $AB=\alpha B+\gamma$.
\item[(iii)] If $\alpha\neq 0=\beta$, then $f_{m,n} = A^m B^n$ for every $m,n\geq 0$ if and only if $AB=\beta A+\gamma$.
\item[(iv)] If $\alpha=\beta= 0$, then $f_{m,n} = A^m B^n$ for every $m,n\geq 0$ if and only if $AB=\gamma$.
\end{itemize}
\end{corollary}

This section is devoted to analyze the asymptotic behavior of the diagonal sequence $\{f_{m,m}\}$. To do so, we will need to consider several different cases separatedly. Throughout the section we will make extensive use of a well-established method due to Pemantle and Wilson that we will refer to as the PW method. The details can be found in \cite{Mel2017, PeWi2004, PeWi2008, PeWi2013, RaWi2008}, for example.

\begin{lemma} 
\label{lem:generatriz}
Let $f(x,y) = \sum_{m,n\geq0} f_{m,n} x^m y^n$ be the generating function of the sequence $\{f_{m,n}\}$ defined in \eqref{eq:recuab1}. Then, if $|x|<1/|A|$ and $|y|<1/|B|$, it holds that
\[
f(x,y) = \frac{1-x-y+ (A+B-AB) x y}{(1-Ax)(1-By)(1-x-y-\gamma x y)} .
\]
\end{lemma}

\begin{proof} 
Just apply Theorem~\ref{th:GFfmn} with $\alpha=\beta=1$.
\end{proof}

\begin{proposition} 
\label{prop:principal} 
If $A, B < 1 + \sqrt{1+\gamma}$ then there is a constant $K$ such that
\[
f_{m,m} \sim \frac{K}{\sqrt{m}}\left( 1 + \sqrt{1+\gamma}\right)^{2 m}
= \frac{K}{\sqrt{m}} \left(2+\gamma+2\sqrt{1+\gamma}\right)^m.
\]
\end{proposition}

\begin{proof} 
According to Lemma~\ref{lem:generatriz}, the generating function of $\{f_{m,n}\}$ can be written as
\[
f(x,y) = \frac{I(x,y)}{1-x-y-\gamma x y},
\]
where $I(x,y)$ is analytic in a neighborhood of the origin.

It is easy to see that the system of algebraic equations
\[
I(x,y) = 0, \qquad x I_x(x,y)-y I_y(x,y) = 0
\] 
has the unique solution
\[
(\overline{x},\overline{y}) 
= \left(\frac{1}{1+\sqrt{1+\gamma}},\frac{1}{1+\sqrt{1+\gamma}}\right). 
\]
This point is a smooth, nondegenerate, isolated, strictly minimal critical point which, furthermore, lies inside the domain of $f(x,y)$ by hypothesis. Under these conditions, the PW method guarantees the existence of a constant $K$ such that
\[ 
f_{m,m} \sim \frac{K}{\sqrt{m}} (\overline{x}\cdot \overline{y})^{-m},
\]
and the result follows.
\end{proof}

We will consider from now on the following sequences:
\begin{equation*}
\mathfrak{p}_{m,n} =
\begin{cases}
0 , & \text{if $m =0$,} \\
A^m , & \text{if $m\geq1, n=0$,} \\
\mathfrak{p}_{m-1,n} + \mathfrak{p}_{m,n-1}+ \gamma \mathfrak{p}_{m-1,n-1}, & \text{if $mn>0$,}
\end{cases}
\end{equation*}
\begin{equation*}
\mathfrak{q}_{m,n} =
\begin{cases}
0 , & \text{if $n =0$,} \\
B^m , & \text{if $n\geq1, m=0$,} \\
\mathfrak{q}_{m-1,n}+ \mathfrak{q}_{m,n-1} + \gamma \mathfrak{q}_{m-1,n-1}, & \text{if $mn>0$,}
\end{cases}
\end{equation*}
and
\begin{equation*}
\mathfrak{r}_{m,n} =
\begin{cases}
1 , & \text{if $n=m =0$,} \\
0 , & \text{if $n =0, m\geq1$,} \\
0 , & \text{if $n\geq1, m=0$,} \\
\mathfrak{r}_{m-1,n}+ \mathfrak{r}_{m,n-1}+ \gamma \mathfrak{r}_{m-1,n-1}, & \text{if $mn>0$.}
\end{cases}
\end{equation*}
Observe that, if $A$ or $B$ are null, then $\mathfrak{p}_{m,n}$ or $\mathfrak{q}_{m,n}$ are, respectively, null sequences. The following lemma shows by we are interested in these new sequences.

\begin{lemma} 
\label{lem:fpqr}
With all the previous notation, we have that
\[
f_{m,n} = \mathfrak{p}_{m,n} + \mathfrak{q}_{m,n} + \mathfrak{r}_{m,n}, 
\quad m,n \ge 0.
\]
\end{lemma}

\begin{proof}
It follows inductively.
\end{proof}

\begin{lemma} 
\label{lem:3fpqr}
Let $f_{\mathfrak{p}}(x,y)$, $f_{\mathfrak{q}}(x,y)$, and $f_{\mathfrak{r}}(x,y)$ be the generating functions of the sequences $\{\mathfrak{p}_{m,n}\}$, $\{\mathfrak{q}_{m,n}\}$, and $\{\mathfrak{r}_{m,n}\}$, respectively. Then,
\begin{itemize}
\item If $|x|<1/|A|$, it holds that $f_{\mathfrak{p}}(x,y) = \frac{Ax - A x^2}{(1-Ax)(1-x-y-\gamma x y)}$.
\item If $|y|<1/|B|$, it holds that $f_{\mathfrak{q}}(x,y) = \frac{Bx - B x^2}{(1-Ax)(1-x-y-\gamma x y)}$. 
\item For every $x,y$, it holds that $f_{\mathfrak{r}}(x,y) = \frac{1-x-y}{1-x-y-\gamma x y}$.
\end{itemize}
\end{lemma}

\begin{proof} 
See Lemma~\ref{lem:generatriz}.
\end{proof}

\begin{proposition} 
\label{prop:PQR} 
With all the previous notation, we have:
\begin{itemize}
\item If $A<1+\sqrt{1+\gamma}$, then there exists a constant $P$ such that 
\[ 
\mathfrak{p}_{m,m} \sim \frac{P}{\sqrt{m}} \left(1+\sqrt{1+\gamma}\right)^{2m}.
\]
\item If $B<1+\sqrt{1+\gamma}$, then there exists a constant $Q$ such that 
\[ 
\mathfrak{q}_{m,m} \sim \frac{Q}{\sqrt{m}} \left(1+\sqrt{1+\gamma}\right)^{2m}.
\]
\item There exists a constant $R$ such that
\[ 
\mathfrak{r}_{m,m} \sim \frac{R}{\sqrt{m}} \left(1+\sqrt{1+\gamma}\right)^{2m}.
\]
\end{itemize}

\end{proposition}

\begin{proof}
It is enough to proceed exactly as in Proposition~\ref{prop:principal}, taking into account the generating functions obtained in Lemma~\ref{lem:3fpqr} and observing that, in each case the critical points lie in the domain of the corresponding generating function so that the PW method can be used.
\end{proof}

For $A\neq 1$, we define the following sequences:
\[
S_{m,n} =
\begin{cases}
0 , & \text{if $m =0$,} \\
1 , & \text{if $m\geq1, n=0$,} \\
S_{m-1,n} + S_{m,n-1} + \gamma S_{m-1,n-1}, & \text{if $mn>0$,}
\end{cases}
\]
\[
\mathfrak{t}_{m,n} =
\begin{cases}
\left( \frac{A+\gamma}{A-1}\right )^n , & \text{if $m n =0$,} \\
\mathfrak{t}_{m-1,n} + \mathfrak{t}_{m,n-1} + \gamma \mathfrak{t}_{m-1,n-1}, & \text{if $mn>0$,}
\end{cases}
\]
and
\[ 
G_{m,n} = A^m \left( \frac{A+\gamma}{A-1}\right )^n, \quad m,n\geq 0.
\]
Since $A \cdot \frac{A+\gamma}{A-1} = A + \frac{A+\gamma}{A-1} + \gamma$, Proposition~\ref{prop:AmBn} implies that $G_{m,n} = G_{m-1,n}+ G_{m,n-1}+ \gamma G_{m-1,n-1}$. Furthermore, we have the following decomposition.

\begin{lemma} 
\label{lem:pSGt}
If $A\neq1$, then
\[
\mathfrak{p}_{m,n}=S_{m,n}+G_{m,n}- \mathfrak{t}_{m,n}, \quad m,n\geq 0.
\]
\end{lemma}

\begin{proof}
Since the four sequences involved satisfy the same recurrence, it is enough to observe that the equality holds for $m n = 0$, which is trivially verified.
\end{proof}

\begin{proposition}
\label{prop:S}
There exists a constant $S$ such that
\[ 
S_{m,m} \sim \frac{S}{\sqrt{m}} \left( 1+ \sqrt{1+\gamma}\right)^{2 m}.
\]
\end{proposition}

\begin{proof}
It is easily seen that, if $|x|<1$, the generating function of $S_{m,n}$ satisfies that 
\[
 f_S (x,y)=\sum_{m,n\geq0}S_{m,n} x^m y^n = \frac{x}{1-x-y-\gamma x y}.
\]
Then, it is enough to argue as in Proposition~\ref{prop:principal} noting that $\frac{1}{1+\sqrt{1+\gamma}}<1$, so that the critical point is in the domain of the generating function and the PW method can be applied.
\end{proof}

\begin{proposition} 
\label{prop:T}
If $A>1+\sqrt{1+\gamma}$, then there exists a constant $T$ such that
\[ 
\mathfrak{t}_{m,m} \sim \frac{T}{\sqrt{m}} \left( 1+ \sqrt{1+\gamma}\right)^{2 m}.
\]
\end{proposition}

\begin{proof}
The generating function of $\mathfrak{t}_{m,n}$ is 
\[
f_\mathfrak{t} (x,y)=\sum_{m,n\geq0} \mathfrak{t}_{m,n} x^m y^n 
= \frac{1-y}{(1-x-y-\gamma x y)\left(1-y \frac{A+\gamma}{A-1}\right)}, 
\]
provided $|y|< \frac{A-1}{A+\gamma}$.

We will proceed as in previous propositions. In this case, in order to be able to apply the PW method, we have to check that $ \frac{1}{1+\sqrt{1+\gamma}} < \frac{A-1}{A-\gamma}$:
\[
A > 1+ \sqrt{1+\gamma} 
\;\Rightarrow\; \frac{A+\gamma}{A-1}< 1+ \sqrt{1+\gamma} 
\;\Rightarrow\; \frac{1}{1+\sqrt{1+\gamma}}< \frac{A-1}{A-\gamma}.
\]
Thus, we apply the usual reasoning and the result follows.
\end{proof}

\begin{proposition}
If $A>1+\sqrt{1+\gamma}$, then 
\[
\mathfrak{p}_{m,m}\sim A^m \left( \frac{A+\gamma}{A-1}\right)^m =G_{m,m}.
\]
\end{proposition}

\begin{proof}
In the first place, let us observe that, if $A \neq 1+\sqrt{1+ \gamma}$, then $ A\left(\frac{A+\gamma}{A-1} \right) > (1+\sqrt{1+\gamma})^2$. Then, it is enough to apply Lemma~\ref{lem:pSGt}, Propositions~\ref{prop:S} and~\ref{prop:T}, and divide by $G_{m,m}$.
\end{proof}

\begin{proposition}
If $B>1+\sqrt{1+\gamma}$, then 
\[
\mathfrak{q}_{m,m} \sim B^m \left( \frac{B+\gamma}{B-1}\right)^m.
\]
\end{proposition}

\begin{proof}
Procceed as in the previous proposition, changing the roles of $A$ and~$B$.
\end{proof}

\begin{corollary}
\label{cor:ABgrandes}
If $A,B>1+\sqrt{1+\gamma}$, then the $f_{m,n}$ defined as in \eqref{eq:recuab1} satisfy
\[ 
f_{m,m} \sim
\begin{cases}
A^m \left( \frac{A+\gamma}{A-1}\right )^m , & \text{if $A>B$,} \\
B^m \left( \frac{B+\gamma}{B-1}\right )^m , & \text{if $B>A$.}
\end{cases}
\]
\end{corollary}

\begin{proof}
By Lemma~\ref{lem:fpqr} we have $f_{m,n}=\mathfrak{p}_{m,n}+\mathfrak{q}_{m,n}+\mathfrak{r}_{m,n}$. 
Moreover,
$\mathfrak{p}_{m,m} \sim A^m \left( \frac{A+\gamma}{A-1}\right )^m$, 
$\mathfrak{q}_{m,m}\sim B^m \left( \frac{B+\gamma}{B-1}\right)^m$ 
and
$\mathfrak{r}_{m,m} \sim \frac{R}{\sqrt{m}} \left( 1+\sqrt{1+\gamma} \right)^{2 m}$.

Now, let us suppose $A>B$. Because $A,B > \left( 1+\sqrt{1+\gamma} \right)$, we also have $AB>A+B+\gamma$ and then it follows that $A \left(\frac{A+\gamma}{A-1}\right)> B \left(\frac{B+\gamma}{B-1}\right)$.
We also know that
\[ 
A \left(\frac{A+\gamma}{A-1}\right) > \left( 1+ \sqrt{1+\gamma}\right)^ 2.
\]
Then,
\[
\frac{f_{m,m}}{A^m \left(\frac{A+\gamma}{A-1}\right)^m}
= \frac{\mathfrak{p}_{m,m}}{A^m \left(\frac{A+\gamma}{A-1}\right)^m}
+ \frac{\mathfrak{r}_{m,m}}{A^m \left(\frac{A+\gamma}{A-1}\right)^m} 
\]
and it is enough to take limits to obtain the desired result.

If $B>A$, the reasoning is analogous.
\end{proof}

\begin{corollary} 
If $A=B>1+\sqrt{1+\gamma}$, then $f_{m,m} \sim 2 A ^m \left(\frac{A+\gamma}{A-1}\right)^m$.
\end{corollary}

\begin{proof}
Reason as in Corollary~\ref{cor:ABgrandes}.
\end{proof}

\begin{proposition} 
If $B<1+\sqrt{1+\gamma} < A$, then $f_{m,m} \sim A^m \left(\frac{A+\gamma}{A-1}\right)^m$.

If $A<1+\sqrt{1+\gamma} < B$, then $f_{m,m} \sim B^m \left(\frac{B+\gamma}{B-1}\right)^m$.
\end{proposition}

\begin{proof}
Let us analyze only the first case, the second is identical.
Lemma~\ref{lem:fpqr} implies that $f_{m,m}=\mathfrak{p}_{m,m}+\mathfrak{q}_{m,m}+ \mathfrak{r}_{m,m}$; 
moreover, using Proposition~\ref{prop:PQR} we have
\begin{equation*}
\mathfrak{p}_{m,m} \sim A^m \left(\frac{A+\gamma}{A-1}\right)^m, 
\qquad
\mathfrak{q}_{m,m} \sim \frac{Q}{\sqrt{m}} \left(1+\sqrt{1+\gamma}\right)^{2 m}, 
\qquad
\mathfrak{r}_{m,m} \sim \frac{R}{\sqrt{m}} \left(1+\sqrt{1+\gamma}\right)^{2 m}.
\end{equation*}
Thus, it is enough to divide $f_{m,m} = \mathfrak{p}_{m,m} + \mathfrak{q}_{m,m} + \mathfrak{r}_{m,m}$ by $A^m \left(\frac{A+\gamma}{A-1}\right)^m$ and take limits when $m \to \infty$, taking into account that
$A \left(\frac{A+\gamma}{A-1}\right) > \left(1+ \sqrt{1+\gamma}\right)^2$.
\end{proof}

We can summarize the previous propositions in the following result, where, without loss of generality, we are assuming $A\leq B$:

\begin{theorem} 
\label{th:diagonalab1}
Let $0\leq A\leq B$, $\gamma>0$ and $f_{m,n}$ defined as in \eqref{eq:recuab1}. Then,
\[ 
f_{m,m} \sim
\begin{cases}
 \frac{K}{\sqrt{m}} \left(1+\sqrt{1+\gamma}\right)^{2m} , 
 & \text{if\, $A\leq B< 1+\sqrt{1+\gamma}$}, \\
  B^m \left( \frac{B+\gamma}{B-1}\right)^m , 
 & \text{if\, $1+\sqrt{1+\gamma} \leq B$}, \\
  2 B^m \left( \frac{B+\gamma}{B-1}\right)^m , 
 & \text{if\, $1+\sqrt{1+\gamma} < A=B$}.
\end{cases}
\]
\end{theorem}

Now, let us analyze the cases $A = 1+\sqrt{1+\gamma}$ or $B = 1+\sqrt{1+\gamma}$. Notice that he case $A=B=1+\sqrt{1+\gamma}$ is already included in Proposition~\ref{prop:AmBn}.

\begin{corollary} 
If $A=B=1+\sqrt{1+\gamma}$ then $f_{m,n}= (1+ \sqrt{1+\gamma})^{m+n}$ and, in particular, $f_{m,m}=\left(1+\sqrt{1+\gamma}\right)^{2m}$.
\end{corollary}

\begin{proof}
Use $\left(1+\sqrt{1+\gamma}\right)^2 = 2 (1+\sqrt{1+\gamma})+\gamma$ and apply Proposition~\ref{prop:AmBn}.
\end{proof}

\begin{remark} 
If $A= 1+\sqrt{1+\gamma}$ then $\frac{A+\gamma}{A-1}=A$ so $A\frac{A+\gamma}{A-1}= A^2=2 +\gamma + 2 (1+\sqrt{1+\gamma})$. Of course, the same can be said for~$B$.
\end{remark}

\begin{proposition} 
If $A = 1+\sqrt{1+\gamma}$ and $B> 1+\sqrt{1+\gamma}$, 
then $f_{m,m} \sim B^m \left(\frac{B+\gamma}{B-1}\right)^m$.
\end{proposition}

\begin{proof} 
Becase $B > 1+\sqrt{1+\gamma}$, we can choose $\eps>0$ such that $B > 1+\sqrt{1+\gamma}+\eps$ and $\sqrt{1+\gamma}-\eps \geq 0$. Now, let us denote $A^{\eps} = 1+\sqrt{1+\gamma}+\eps$ and $A^{-\eps} = 1+\sqrt{1+\gamma}-\eps$, and take $f_{m,n}^{\eps}$, $f^{-\eps}_{m,n}$ the corresponding sequences defined using $A^{\eps}$ and~$A^{-\eps}$.

Because $B > A^{\eps} > 1+\sqrt{1+\gamma}$, Theorem~\ref{th:diagonalab1} show that $f_{m,n}^{\eps} \sim B^m (\frac{B+\gamma}{B-1})^m $, that does not depend on $\eps$. On the other hand, using $B > 1+\sqrt{1+\gamma} > A^{-\eps}$, Theorem~\ref{th:diagonalab1} again gives $f^{-\eps}_{m,n} \sim B^m (\frac{B+\gamma}{B-1})^m$, also independent on~$\eps$. New, let us take $\eps \to 0$.
\end{proof}

\begin{proposition} 
If $B= 1+\sqrt{1+\gamma}$ and $A> 1+\sqrt{1+\gamma}$, 
then $f_{m,m} \sim A^m \left(\frac{A+\gamma}{A-1}\right)^m$.
\end{proposition}

\begin{proof} 
Proceed as in the previous proposition, changing the roles of $A$ and $B$.
\end{proof}

\begin{proposition} 
If $A= 1+\sqrt{1+\gamma}$ and $B< 1+\sqrt{1+\gamma}$ or vice versa, then 
\[
f_{m,m} \sim (1+\sqrt{1+\gamma}) ^{2 m}.
\]
\end{proposition}

\begin{proof}
By Lemma~\ref{lem:fpqr}, $f_{m,m}=\mathfrak{p}_{m,m}+ \mathfrak{q}_{m,m}+ \mathfrak{r}_{m,m}$, and, by Proposition~\ref{prop:PQR}, $\mathfrak{q}_{m,m} \sim \frac{Q}{\sqrt{m}} (1+\sqrt{1+\gamma})^{2 m}$ and $\mathfrak{r}_{m,m} \sim \frac{R}{\sqrt{m}} (1+\sqrt{1+\gamma})^{2 m}$. Now
\[
\mathfrak{p}_{m,n} =
\begin{cases}
0 , & \text{if $m =0$,} \\
(1+\sqrt{1+\gamma})^m , 
& \text{if $m\geq1, n=0$,} \\
\mathfrak{p}_{m-1,n} + \mathfrak{p}_{m,n-1}+ \gamma \mathfrak{p}_{m-1,n-1}, 
& \text{if $mn>0$,}
\end{cases}
\]
and we want to study its asymptotic behavior. We will see that $\mathfrak{p}_{m,n} \sim (1+\sqrt{1+\gamma})^{m+n}$ and, from this, the result follows easily.

If $n = 0$, it is evident from the definition of $\mathfrak{p}_{m,n}$.
If $n = 1$, it is easy to see inductively that
\begin{align*} 
\mathfrak{p}_{m,1} &= \mathfrak{p}_{m,0} + (1+\gamma) \sum_{i=2}^{m-1} \mathfrak{p}_{i,0}
+ \mathfrak{p}_{1,0} \\
&= {(1+\sqrt{1+\gamma})}^m + (1+\gamma) {(1+\sqrt{1+\gamma})}^2\,
\frac{(1+\sqrt{1+\gamma}) ^{m-2}-1}{\sqrt{1+\gamma}} + (1+\sqrt{1+\gamma}) \\
&= {(1+\sqrt{1+\gamma})}^{m+1}+ (1+\sqrt{1+\gamma}) - \sqrt{1+\gamma} {(1+\sqrt{1+\gamma})}^2.
\end{align*}
To conclude, it is enough to reason by induction on $n$.
\end{proof}

To finish we have to calculate the value of the constant $K$ in Theorem~\ref{th:diagonalab1}. For this we use the Meltzer methodology, see \cite[Th.~54]{Mel2017}; although there it is a rational function, the truth is that it is still valid for quotients of holomorphic functions in an environment of the origin (see \cite[Th.~3.5]{PeWi2002}). 

To compute $K$, let's remember that we are in the case $A,B<1+ \sqrt{1+\gamma}$ and that, in this case, the generating function is given by
\[ 
f(x,y) = \frac{\frac{1-x-y+(A+B-AB)xy}{(1-Ax)(1-By)}}{1-x-y+\gamma x y} =: \frac{I(x,y)}{J(x,y)},
\]
where both $I$ and $J$ are holomorphic functions in an environment of the origin. Furthermore, the only strictly minimal point, that is also isolated, smooth and non-degenerate, is
\[ 
(\omega_1,\omega_2) = \left( \frac{1}{1+\sqrt{1+\gamma}}, \frac{1}{1+\sqrt{1+\gamma}}\right).
\]

In this context, we have (see \cite[Th.~54]{Mel2017}, that gives precise values for the constants) 
\[
K = \frac{1}{\sqrt{2 \pi}}\frac{1}{\sqrt{\mathcal{H}}} \,C_0,
\]
where 
\[
C_0 = \frac{-I(\omega_1,\omega_2)}{\omega_2 J_y(\omega_1,\omega_2)} 
\quad\text{ and }\quad
\mathcal{H} = 2+ \lambda^{-1}(U_{1,1}-2U_{1,2}+U_{2,2}),
\]
with
$\lambda = \omega_2 J_y(\omega_1,\omega_2) = \omega_1 I_x(\omega_1,\omega_2)$ and
\[ 
U_{1,1} = \omega_1^2J_{xx} (\omega_1,\omega_2), 
\quad
U_{1,2} = \omega_1 \omega_2 J_{xy} (\omega_1,\omega_2),
\quad
U_{2,2} = \omega_2^2J_{yy} (\omega_1,\omega_2).
\]
With these ingredients, we are able to calculate the constant $K$ in Theorem~\ref{th:diagonalab1}.

\begin{proposition} 
The constant $K$ in Theorem~\ref{th:diagonalab1} is
\[
K = \frac{1}{2 \sqrt{\pi }}\frac{\gamma ( A+B- AB+\gamma )}{\sqrt[4]{\gamma +1} 
\left(A B \sqrt{\gamma +1} - AB + \gamma \left(-A-B+\sqrt{\gamma +1}+1\right)\right)}.
\]
\end{proposition}

\begin{proof}
First, let us observe that $J_x = -1-\gamma y$, $J_y = -1-\gamma x$, $J_{xy} = - \gamma$, $J_{xx}=J_{yy}=0$. 
Consequently, $U_{1,1}=U_{2,2}=0$ and $\mathcal{H}=2-\frac{2}{\lambda} U_{1,2}$. 
Moreover,
\[ 
\lambda = \omega_2 J_y(\omega_1,\omega_2)
= \frac{1}{1+\sqrt{1+\gamma}} \left(-1- \frac{\gamma}{1+\sqrt{1+\gamma}}\right),
\]
\[
U_{1,2} = \omega_1 \omega_2 J_{xy}(\omega_1,\omega_2)
= \left( \frac{1}{1+\sqrt{1+\gamma}}\right)^2 \cdot(-\gamma),
\]
and therefore $\mathcal{H} = {2}/{\sqrt{1+\gamma}}$.

On the other hand, $C_0 = -I(\omega_1,\omega_2) \lambda^{-1}$. Since
\[
 I(\omega_1,\omega_2) = \frac{\gamma+A+B-AB}{(\sqrt{1+\gamma}-A+1)(\sqrt{1+\gamma}-B+1)},
\]
it follows that
\[
C_0 = \frac{\gamma (A+B-AB+\gamma)}{(\sqrt{1+\gamma}-1)\cdot(\sqrt{1+\gamma}-A+1)
\cdot (\sqrt{1+\gamma}-B+1)},
\]
and we are done.
\end{proof}

Now, let us forget the condition $\alpha=1=\beta$ what we assumed, without loss of generality, at the beginning of Section~\ref{sec:asympfnn}. To do so, it is enough to replace $A$ by $A/\alpha$, $B$ by $B/\beta$, $\gamma$ by $\frac{\gamma}{\alpha \beta}$ and $f_{m,n}$ by $\alpha^{-m} \beta^{-n} f_{m,n}$ in the previous results. If, in addition, we rewrite $K$ in a more symmetric fashion, we get the following.

\begin{theorem} 
\label{th:diagonal}
Let $A,B \geq 0$ and $\alpha,\beta,\gamma\geq 0$ with $\alpha\beta\neq 0$. If $\{f_{m,n}\}$ is defined as in \eqref{eq:recu}, then
\[
f_{m,m} \sim
\begin{cases}
\frac{K \alpha^m \beta^m}{\sqrt{m}} \left(1+\sqrt{1+\frac{\gamma}{\alpha \beta}}\right)^{2 m} , 
& \text{if\, $\frac{A}{\alpha},\frac{B}{\beta} < 1+\sqrt{1+\frac{\gamma}{\alpha \beta}}$},
\\
B^m \left(\frac{\alpha B+\gamma }{B- \beta }\right)^m , 
& \text{if\, $\frac{A}{\alpha} \leq 1+\sqrt{1+\frac{\gamma}{\alpha \beta}} \leq \frac{B}{\beta}$}, 
\\
B^m \left(\frac{\alpha B+\gamma }{B- \beta }\right)^m , 
& \text{if\, $1+\sqrt{1+\frac{\gamma}{\alpha \beta}} \leq \frac{A}{\alpha} < \frac{B}{\beta}$}, 
\\
A^m \left(\frac{A \beta +\gamma }{A -\alpha}\right)^m , 
& \text{if\, $\frac{B}{\beta} \leq 1+\sqrt{1+\frac{\gamma}{\alpha \beta}} \leq\frac{A}{\alpha}$}, 
\\
A^m \left(\frac{A \beta +\gamma }{A -\alpha}\right)^m , 
& \text{if\, $1+\sqrt{1+\frac{\gamma}{\alpha \beta}} \leq \frac{B}{\beta} < \frac{A}{\alpha}$}, 
\\
2 B^m \left(\frac{\alpha B+\gamma }{B- \beta }\right)^m , 
& \text{if\, $1+\sqrt{1+\frac{\gamma}{\alpha \beta}} < \frac{A}{\alpha} = \frac{B}{\beta}$},
\end{cases}
\]
with 
\[
K = \frac{\gamma (A \beta + B \alpha - AB + \gamma)}
 {2 \sqrt{\pi} \sqrt[4]{\frac{\gamma}{\alpha \beta}+1} 
 \left(AB \alpha\beta \left(\sqrt{\frac{\gamma }{\alpha \beta }+1} - 1 \right) 
 + \alpha\beta\gamma \left(\sqrt{\frac{\gamma }{\alpha \beta }+1} + 1 \right)
 - (A\beta+B\alpha) \gamma \right)}.
\]
\end{theorem}

As we discussed at the beginning of this section, we have only been considering the case $\alpha\beta\neq 0$. However, the remaining cases may be approached in a similar fashion, and we obtain the following general result.

\begin{theorem}
\label{th:diagonal-0}
Let $A,B,\alpha,\beta,\gamma\ge0$, and $f_{m,n}$ the sequence defined in \eqref{eq:recu}. Then,
\[
f_{m,m} \sim
\begin{cases}
\frac{K}{\sqrt{m}} \left(\sqrt{\alpha \beta}+\sqrt{\alpha\beta+\gamma})\right)^{2 m} , 
& \text{if\, $A\beta,B\alpha < \alpha\beta+\sqrt{\alpha\beta(\alpha\beta+\gamma)}$},
\\
B^m \left(\frac{\alpha B+\gamma }{B- \beta }\right)^m , 
& \text{if\, $A\beta \leq \alpha\beta+\sqrt{\alpha\beta(\alpha\beta+\gamma)} \leq B\alpha$}, 
\\
B^m \left(\frac{\alpha B+\gamma }{B- \beta }\right)^m , 
& \text{if\, $\alpha\beta+\sqrt{\alpha\beta(\alpha\beta+\gamma)} \leq A\beta < B\alpha$}, 
\\
A^m \left(\frac{A \beta +\gamma }{A -\alpha}\right)^m , 
& \text{if\, $B\alpha \leq \alpha\beta+\sqrt{\alpha\beta(\alpha\beta+\gamma)} \leq A\beta$}, 
\\
A^m \left(\frac{A \beta +\gamma }{A -\alpha}\right)^m , 
& \text{if\, $\alpha\beta+\sqrt{\alpha\beta(\alpha\beta+\gamma)} \leq B\alpha < A\beta$}, 
\\
2 B^m \left(\frac{\alpha B+\gamma }{B- \beta }\right)^m , 
& \text{if\, $\alpha\beta+\sqrt{\alpha\beta(\alpha\beta+\gamma)} < A\beta = B\alpha$},
\end{cases}
\]
with $K$ as in Theorem~\ref{th:diagonal}.
\end{theorem}

Finally, as a direct consequence of Theorem~\ref{th:diagonal-0}, we have the following result regarding the behavior of $f_{m+1,m+1}/f_{m,m}$.

\begin{theorem}
\label{th:cocientes}
Let $A,B,\alpha,\beta,\gamma\geq 0$. If $\{f_{m,n}\}$ is defined as in \eqref{eq:recu}, then the limit
\[
 \mathfrak{L}: = \lim_{m\to\infty} \frac{f_{m+1,m+1}}{f_{m,m}}
\]
always exists, and its value is
\[
\mathfrak{L} =
\begin{cases}
\left(\sqrt{\alpha\beta} + \sqrt{\alpha\beta+\gamma}\right)^2 , 
& \text{if\, $A\beta,B\alpha < \alpha\beta+\sqrt{\alpha\beta(\alpha\beta+\gamma)}$}, 
\\
B \left(\frac{\alpha B+\gamma}{B-\beta}\right) , 
& \text{if\, $A\beta \leq \alpha\beta+\sqrt{\alpha\beta(\alpha\beta+\gamma)} \leq B\alpha$}, 
\\
B \left(\frac{\alpha B+\gamma }{B-\beta}\right) , 
& \text{if\, $\alpha\beta+\sqrt{\alpha\beta(\alpha\beta+\gamma)} \leq A\beta < B\alpha$}, 
\\
A \left(\frac{A \beta + \gamma }{A-\alpha}\right) , 
& \text{if\, $B\alpha \leq \alpha\beta+\sqrt{\alpha\beta(\alpha\beta+\gamma)} \leq A\beta$},
\\
A \left(\frac{A \beta +\gamma }{A-\alpha}\right) , 
& \text{if\, $\alpha\beta+\sqrt{\alpha\beta(\alpha\beta+\gamma)} \leq B\alpha < A\beta$}, 
\\
B \left(\frac{\alpha B+\gamma }{B- \beta }\right) , 
& \text{if\, $\alpha\beta+\sqrt{\alpha\beta(\alpha\beta+\gamma)} < A\beta = B\alpha$}.
\end{cases}
\]
\end{theorem}


\section{P-recursivity of central weighted Delannoy numbers}
\label{sec:recur}

In this section we will still assume that $\alpha=\beta=1$, i.e., we consider the sequence $\{f_{m,n}\}$ defined as in \eqref{eq:recuab1}. Recall that this is not a restriction as long as $\alpha\beta\neq 0$. The remaining cases either are trivial (if $\alpha=\beta=0$), or admit a similar treatment.

First, we will see that the generating function of the sequence $\{\mathfrak{f}_n\}=\{f_{n,n}\}$ is holonomic (D-finite). We will do so constructively, i.e., explicitly exhibiting a differential equation satisfied by the generating function. 

\begin{proposition} 
\label{prop:D-finita} 
The generating function $G(z) = \sum_{n\ge0} \mathfrak{f}_{n} z^n$ satisfies a differential equation
\begin{equation}
\label{eq:ED}
q_0(z) G(z) + z q_1(z) G'(z) + z^2 q_2(z) G''(z) = c(z),
\end{equation} 
where $q_0(z)$, $q_1(z)$ and $q_2(z)$ are polynomials of degree at most $4$, $c(z)$ is a polynomial of degree at most $2$, and all their coefficients depend on $A$, $B$ and $\gamma$.
\end{proposition}

\begin{proof}
Let us consider the polynomials
\[
  q_i(z) = q_{i,0} + q_{i,1}z + q_{i,2}z^2 + q_{i,3}z^3 + q_{i,4}z^4,
  \quad i=0,1,2, \quad\text{and}\quad
  c(z) = c_0 + c_1z + c_2z^2
\]
with
\begin{align*}
q_{0,0} &= - (-1 + A) (-1 + B) (20 A B - 10 A^2 B - 10 A B^2 + 4 A^2 B^2 
  + 10 A \gamma - 5 A^2 \gamma + 10 B \gamma \\*
  &\qquad - A^2 B \gamma - 5 B^2 \gamma 
  - A B^2 \gamma + 6 \gamma^2 - A \gamma^2 - B \gamma^2),
\\
q_{0,1} &= -2 (-1 + A) (-1 + B) (-2 + A + B) (6 A B + 3 A \gamma + 3 B \gamma 
  + 6 A B \gamma + 2 \gamma^2 + 2 A \gamma^2 \\*
  &\qquad + 2 B \gamma^2 + A B \gamma^2 + \gamma^3),
\\
q_{0,2} &= -8 A^3 B + 4 A^4 B + 12 A^3 B^2 - 8 A^4 B^2 - 8 A B^3 + 12 A^2 B^3  
  - 4 A^3 B^3 + 2 A^4 B^3 + 4 A B^4 - 8 A^2 B^4 \\
  &\qquad  + 2 A^3 B^4 - 4 A^3 \gamma  
  + 2 A^4 \gamma - 12 A^2 B \gamma + 4 A^3 B \gamma - 2 A^4 B \gamma 
  - 12 A B^2 \gamma + 36 A^2 B^2 \gamma - 4 A^3 B^2 \gamma \\
  &\qquad - 5 A^4 B^2 \gamma - 4 B^3 \gamma + 4 A B^3 \gamma - 4 A^2 B^3 \gamma 
  + 4 A^3 B^3 \gamma + A^4 B^3 \gamma + 2 B^4 \gamma - 2 A B^4 \gamma - 5 A^2 B^4 \gamma \\
  &\qquad + A^3 B^4 \gamma - 8 A^2 \gamma^2 + 2 A^3 \gamma^2 - 4 A B \gamma^2 
  + 4 A^2 B \gamma^2 - 4 A^3 B \gamma^2 - A^4 B \gamma^2 - 8 B^2 \gamma^2 + 4 A B^2 \gamma^2  \\
  &\qquad + 28 A^2 B^2 \gamma^2 - 6 A^3 B^2 \gamma^2 - A^4 B^2 \gamma^2 + 2 B^3 \gamma^2 
  - 4 A B^3 \gamma^2 - 6 A^2 B^3 \gamma^2 + 4 A^3 B^3 \gamma^2 - A B^4 \gamma^2 \\
  &\qquad - A^2 B^4 \gamma^2 - 2 A \gamma^3 - 5 A^2 \gamma^3 + A^3 \gamma^3 
  - 2 B \gamma^3 + 4 A B \gamma^3 + 6 A^2 B \gamma^3 - 4 A^3 B \gamma^3 - 5 B^2 \gamma^3 \\
  &\qquad + 6 A B^2 \gamma^3 + 4 A^2 B^2 \gamma^3 + B^3 \gamma^3 
  - 4 A B^3 \gamma^3 - A \gamma^4 - A^2 \gamma^4 
  - B \gamma^4 + 4 A B \gamma^4 - B^2 \gamma^4,
\\
q_{0,3} &= 2 A B (A + \gamma) (B + \gamma) (2 A B + A \gamma + B \gamma) 
  (6 - 3 A - 3 B + 2 A B + 6 \gamma 
  - 2 A \gamma - 2 B \gamma + A B \gamma + \gamma^2),
\\
q_{0,4} &= -A B \gamma^2 (A + \gamma) (B + \gamma) (20 A B - 10 A^2 B - 10 A B^2 + 6 A^2 B^2 
  + 10 A \gamma - 5 A^2 \gamma + 10 B \gamma \\*
  &\qquad + A^2 B \gamma - 5 B^2 \gamma + A B^2 \gamma 
  + 4 \gamma^2 + A \gamma^2 + B \gamma^2),
\\
q_{1,0} &= (-1 + A) (-1 + B) (20 A B - 10 A^2 B - 10 A B^2 + 4 A^2 B^2 
  + 10 A \gamma - 5 A^2 \gamma + 10 B \gamma - A^2 B \gamma \\*
  &\qquad - 5 B^2 \gamma - A B^2 \gamma + 6 \gamma^2 - A \gamma^2 - B \gamma^2),
\\
q_{1,1} &= -32 A B + 48 A^2 B - 12 A^3 B - 2 A^4 B + 48 A B^2 - 68 A^2 B^2 + 
  + 14 A^3 B^2 + 2 A^4 B^2 - 12 A B^3 \\
  &\qquad  + 14 A^2 B^3 - 2 A B^4 + 2 A^2 B^4 - 16 A \gamma + 24 A^2 \gamma 
  - 6 A^3 \gamma - A^4 \gamma - 16 B \gamma 
  + 16 A B \gamma + 10 A^2 B \gamma \\
  &\qquad  - 8 A^3 B \gamma + 24 B^2 \gamma + 10 A B^2 \gamma  
  - 60 A^2 B^2 \gamma + 18 A^3 B^2 \gamma 
  + A^4 B^2 \gamma - 6 B^3 \gamma - 8 A B^3 \gamma + 18 A^2 B^3 \gamma \\
  &\qquad - B^4 \gamma + A^2 B^4 \gamma - 12 \gamma^2 + 8 A \gamma^2 
  + 12 A^2 \gamma^2 - 6 A^3 \gamma^2 + 8 B \gamma^2 - 14 A^2 B \gamma^2 
  + 4 A^3 B \gamma^2 + 12 B^2 \gamma^2 \\
  &\qquad  - 14 A B^2 \gamma^2 - 4 A^2 B^2 \gamma^2 + 4 A^3 B^2 \gamma^2 
  - 6 B^3 \gamma^2 + 4 A B^3 \gamma^2 
  + 4 A^2 B^3 \gamma^2 - 6 \gamma^3 + 10 A \gamma^3 - 3 A^2 \gamma^3 \\
  &\qquad + 10 B \gamma^3 - 16 A B \gamma^3 + 4 A^2 B \gamma^3 
  - 3 B^2 \gamma^3 + 4 A B^2 \gamma^3,
\\
q_{1,2} &= 32 A^3 B - 16 A^4 B - 48 A^3 B^2 + 28 A^4 B^2 + 32 A B^3  
  - 48 A^2 B^3 + 20 A^3 B^3  - 6 A^4 B^3 - 16 A B^4 \\
  &\qquad + 28 A^2 B^4  
  - 6 A^3 B^4 - 4 A^4 B^4 + 16 A^3 \gamma - 8 A^4 \gamma + 48 A^2 B \gamma 
  - 16 A^3 B \gamma + 4 A^4 B \gamma + 48 A B^2 \gamma \\
  &\qquad - 144 A^2 B^2 \gamma + 14 A^3 B^2 \gamma + 19 A^4 B^2 \gamma 
  + 16 B^3 \gamma - 16 A B^3 \gamma + 14 A^2 B^3 \gamma 
  - 4 A^3 B^3 \gamma - 11 A^4 B^3 \gamma \\
  &\qquad - 8 B^4 \gamma + 4 A B^4 \gamma 
  + 19 A^2 B^4 \gamma - 11 A^3 B^4 \gamma  + 20 A^2 \gamma^2 + 4 A^3 \gamma^2 
  - 4 A^4 \gamma^2 + 28 A B \gamma^2 - 10 A^2 B \gamma^2 \\
  &\qquad - 8 A^3 B \gamma^2 + 9 A^4 B \gamma^2 + 20 B^2 \gamma^2 
  - 10 A B^2 \gamma^2 - 108 A^2 B^2 \gamma^2 + 32 A^3 B^2 \gamma^2 
  - 3 A^4 B^2 \gamma^2 + 4 B^3 \gamma^2 \\
  &\qquad - 8 A B^3 \gamma^2 
  + 32 A^2 B^3 \gamma^2 - 24 A^3 B^3 \gamma^2 - 4 B^4 \gamma^2 
  + 9 A B^4 \gamma^2 - 3 A^2 B^4 \gamma^2 + 2 A \gamma^3 + 17 A^2 \gamma^3 \\
  &\qquad - 3 A^3 \gamma^3 + 2 B \gamma^3 + 20 A B \gamma^3 
  - 40 A^2 B \gamma^3 + 12 A^3 B \gamma^3 + 17 B^2 \gamma^3  
  - 40 A B^2 \gamma^3 - 6 A^3 B^2 \gamma^3 \\
  &\qquad - 3 B^3 \gamma^3 + 12 A B^3 \gamma^3 - 6 A^2 B^3 \gamma^3 + A \gamma^4 
  + 3 A^2 \gamma^4 + B \gamma^4 - 6 A^2 B \gamma^4  
  + 3 B^2 \gamma^4 - 6 A B^2 \gamma^4,
\\
q_{1,3} &= 96 A^3 B^3 - 48 A^4 B^3 - 48 A^3 B^4 + 28 A^4 B^4 + 144 A^3 B^2 \gamma  
  - 72 A^4 B^2 \gamma + 144 A^2 B^3 \gamma - 48 A^3 B^3 \gamma \\
  &\qquad + 8 A^4 B^3 \gamma - 72 A^2 B^4 \gamma + 8 A^3 B^4 \gamma 
  + 14 A^4 B^4 \gamma + 28 A^3 B \gamma^2 - 14 A^4 B \gamma^2 + 212 A^2 B^2 \gamma^2 \\
  &\qquad  + 34 A^3 B^2 \gamma^2 - 48 A^4 B^2 \gamma^2  
  + 28 A B^3 \gamma^2 + 34 A^2 B^3 \gamma^2 
  - 32 A^3 B^3 \gamma^2 + 22 A^4 B^3 \gamma^2 - 14 A B^4 \gamma^2 \\
  &\qquad 
  - 48 A^2 B^4 \gamma^2 + 22 A^3 B^4 \gamma^2 - 10 A^3 \gamma^3  
  + 5 A^4 \gamma^3 + 38 A^2 B \gamma^3 + 24 A^3 B \gamma^3 - 18 A^4 B \gamma^3  
  + 38 A B^2 \gamma^3 \\
  &\qquad  + 204 A^2 B^2 \gamma^3 - 62 A^3 B^2 \gamma^3 + 3 A^4 B^2 \gamma^3 - 10 B^3 \gamma^3 
  + 24 A B^3 \gamma^3 - 62 A^2 B^3 \gamma^3 + 32 A^3 B^3 \gamma^3  \\
  &\qquad + 5 B^4 \gamma^3 - 18 A B^4 \gamma^3 + 3 A^2 B^4 \gamma^3 - 14 A^2 \gamma^4 
  + 4 A^3 \gamma^4 + 66 A^2 B \gamma^4 - 20 A^3 B \gamma^4 - 14 B^2 \gamma^4 \\
  &\qquad + 66 A B^2 \gamma^4  + 4 A^2 B^2 \gamma^4 + 4 A^3 B^2 \gamma^4 
  + 4 B^3 \gamma^4 - 20 A B^3 \gamma^4 
  + 4 A^2 B^3 \gamma^4 - 4 A \gamma^5 - A^2 \gamma^5 - 4 B \gamma^5 \\
  &\qquad + 16 A B \gamma^5  
  + 4 A^2 B \gamma^5 - B^2 \gamma^5 + 4 A B^2 \gamma^5,
\\
q_{1,4} &= -A B \gamma^2 (A + \gamma) (B + \gamma) (36 A B - 18 A^2 B - 18 A B^2 
  + 10 A^2 B^2 + 18 A \gamma - 9 A^2 \gamma + 18 B \gamma + A^2 B \gamma \\*
  &\qquad - 9 B^2 \gamma + A B^2 \gamma + 8 \gamma^2 
  + A \gamma^2 + B \gamma^2),
\\
q_{2,0} &= -2 (-1 + A) (-1 + B) (-2 A + A^2 - \gamma) (-2 B + B^2 - \gamma),
\\
q_{2,1} &= 2 (-2 A + A^2 - \gamma) (-2 B + B^2 - \gamma) (4 - 4 A - A^2 - 4 B + 4 A B 
  + A^2 B - B^2 + A B^2 + 2 \gamma \\*
  &\qquad - 3 A \gamma - 3 B \gamma + 4 A B \gamma),
\\
q_{2,2} &= -2 (-2 A + A^2 - \gamma) (-2 B + B^2 - \gamma) (-4 A^2 + 4 A^2 B - 4 B^2 
  + 4 A B^2 + A^2 B^2 - 4 A \gamma - 2 A^2 \gamma \\* 
  &\qquad - 4 B \gamma + 8 A B \gamma + 3 A^2 B \gamma - 2 B^2 \gamma 
  + 3 A B^2 \gamma + \gamma^2 - 3 A \gamma^2 - 3 B \gamma^2 + 6 A B \gamma^2),
\\
q_{2,3} &= 2 (-2 A + A^2 - \gamma) (-2 B + B^2 - \gamma) (4 A^2 B^2 + 4 A^2 B \gamma 
  + 4 A B^2 \gamma + 2 A^2 B^2 \gamma - A^2 \gamma^2 + 4 A B \gamma^2 \\*
  &\qquad + 3 A^2 B \gamma^2 - B^2 \gamma^2 + 3 A B^2 \gamma^2 
  - A \gamma^3 - B \gamma^3 + 4 A B \gamma^3),
\\
q_{2,4} &= -2 A B (-2 A + A^2 - \gamma) (-2 B + B^2 - \gamma) \gamma^2 (A + \gamma) (B + \gamma),
\\
c_0 &= -((-1 + A) (-1 + B) (20 A B - 10 A^2 B - 10 A B^2 + 4 A^2 B^2 
  + 10 A \gamma - 5 A^2 \gamma + 10 B \gamma - A^2 B \gamma \\*
  &\qquad - 5 B^2 \gamma - A B^2 \gamma + 6 \gamma^2 - A \gamma^2 - B \gamma^2)),
\\
c_1 &= -2 (-1 + A) (-1 + B) (-2 + A + B) (6 A B + 3 A \gamma + 3 B \gamma + 6 A B \gamma 
  + 2 \gamma^2 + 2 A \gamma^2 + 2 B \gamma^2 \\*
  &\qquad 
  + A B \gamma^2 + \gamma^3),
\\
c_2 &= (2 - A - B + A B + \gamma) (-2 A^3 B + 2 A^3 B^2 - 2 A B^3 + 2 A^2 B^3 - A^3 \gamma 
  - 3 A^2 B \gamma - 3 A B^2 \gamma  + 6 A^2 B^2 \gamma \\*
  &\qquad 
  + A^3 B^2 \gamma - B^3 \gamma + A^2 B^3 \gamma - 2 A^2 \gamma^2 - 2 B^2 \gamma^2 
  + 4 A^2 B^2 \gamma^2 - A^2 \gamma^3 + A^2 B \gamma^3 - B^2 \gamma^3 + A B^2 \gamma^3).
\end{align*}

Using symbolic computation it can be verified that $G(z)$ satisfies the differential equation of the statement.
\end{proof}

\begin{remark} 
\label{rem:D-finita2} 
Observe that, if $\gamma = (B-2) B$ or $\gamma=(A-2) A$, then $q_{2,i}=0$ for $i=0,\dots,4$ in the proof of Proposition~\ref{prop:D-finita}. Thus, $q_2(z)=0$ and the generating function $G(z) = \sum_{n\ge0} \mathfrak{f}_{n} z^n$ satisfies a differential equation of order~$1$
\begin{equation*}
q_0(z) G(z) + z q_1(z) G'(z) = c(z),
\end{equation*} 
with $q_0(z)$, $q_1(z)$, and $c(z)$ as in Proposition~\ref{prop:D-finita}.
\end{remark}

The previous proposition implies that $\{\mathfrak{f}_{n}\}$ is a P-recursive sequence \cite[Th.~6.4.6]{St1999}. Next, we specify the form of the recurrence satisfied by~$\{\mathfrak{f}_{n}\}$.

\begin{proposition} 
\label{prop:P-recursiva} 
The diagonal sequence $\{\mathfrak{f}_m\}$ is P-recursive. Namely, it satisfies a recurrence relation of the form
\begin{equation}
\label{eq:recgeneral}
  p_0(m) \mathfrak{f}_m + p_1(m) \mathfrak{f}_{m-1} + p_2(m) \mathfrak{f}_{m-2}
  + p_3(m)\mathfrak{f}_{m-3} + p_4(m) \mathfrak{f}_{m-4} = 0,
\end{equation}
where the $p_i(x)$ are polynomials of degree at most~$2$.
\end{proposition}

\begin{proof}
If we differentiate the series $G(z) = \sum_{m \ge 0} \mathfrak{f}_{m} z^{m}$ term by term and we substitute in \eqref{eq:ED}, it can be directly verified that the coefficients $\{\mathfrak{f}_{m}\}_{m=0}^{\infty}$
satisfy the claimed recurrence.
\end{proof}

\begin{remark} 
If we recall Remark~\ref{rem:D-finita2}, when $(B-2) B = \gamma$ or $(A-2) A = \gamma$ the differential equation satisfied by $G(z)$ was of order~$1$. This implies that we have a slightly simpler version of Proposition~\ref{prop:P-recursiva}. In fact, in such case it can be seen that the sequence $\{\mathfrak{f}_m\}$ satisfies a $4$-term recurrence relation 
\begin{equation*}
  p_0(m) \mathfrak{f}_m + p_1(m) \mathfrak{f}_{m-1} + p_2(m) \mathfrak{f}_{m-2}
  + p_3(m)\mathfrak{f}_{m-3} + p_4(m) \mathfrak{f}_{m-4} = 0,
\end{equation*}
where the $p_i(x)$ are polynomials of degree at most~$1$.
\end{remark}

Using the polynomials $q_0(z)$, $q_1(z)$, $q_2(z)$, and $c(z)$ whose coefficients were given in the proof of Proposition~\ref{prop:D-finita}, it is very easy, with the help of any computer algebra system, to find the precise polynomials $p_0(x),\dots,p_4(x)$ that appear in Proposition~\ref{prop:P-recursiva}. However, as suggested by the proof of Proposition~\ref{prop:D-finita}, their expressions are rather cumbersome, so we opt not to include them in the paper.

In addition, it is important to note that, although the results of this sections are stated for the case $\alpha=\beta=1$, undoing the change explained at the beginning of Section~\ref{sec:asympfnn} we would get, as a direct consequence of Propositions~\ref{prop:D-finita} and~\ref{prop:P-recursiva}, the corresponding results for the more general situation $\alpha\beta\neq 0$. Anyhow, the degrees of the involved polynomials do not vary and we do not reproduce their precise expressions for the same reason.

In fact, all the remaining cases ($\alpha\beta=0$) can be approached using the same general strategy. Once the existence of a recurrence relation like \eqref{eq:recgeneral} is guaranteed, for fixed values of $A$, $B$, $\alpha$, $\beta$ and $\gamma$, this recurrence relation \eqref{eq:recgeneral} can be obtained by solving a linear system with $15$ equations built using the values $\mathfrak{f}_0,\mathfrak{f}_1,\dots,\mathfrak{f}_{18}$. 

\begin{remark} 
Although, from a formal point of view, Proposition~\ref{prop:P-recursiva} is deduced from Proposition~\ref{prop:D-finita}, 
they were in fact conceived following the inverse path. The recurrence coefficients for a finite number of terms of the sequence $\{\mathfrak{f}_{m}\}$ were first obtained by symbolic computation, and from them, we built the differential equation of Proposition~\ref{prop:D-finita} (that is, this allowed to guess the coefficients $q_{i,j}$ and $c_j$ that appear in the proof). Once that the differential equation has been ``discovered'', to check that $G(z)$ satisfies the equation is just a simple computational task.
\end{remark}

It is easy to check that $p_0(1)=0$, but this is not a problem to apply the recurrence \eqref{eq:recgeneral}, becase this formula is defined for $m \ge 4$. However, the polynomial $p_0(x)$ can have other positive integer roots. If $p_0(m)\neq 0$ for all $m \ge 4$, \eqref{eq:recgeneral} allows us to express $\mathfrak{f}_m$ as a $4$-term recurrence relation for $m \ge 4$. 
Unfortunately, if $p_0(m_0) = 0$ for some integer $m_0 \ge 4$, we can not isolate $\mathfrak{f}_{m_0}$ in \eqref{eq:recgeneral} and it cannot be computed using the recurrence. In any case, it is still possible to recursively compute $\mathfrak{f}_m$ for every $m\geq \max\{n:p_0(n)=0\}$. The following examples illustrate these two possible situations.

\begin{example}
If $A=5,B=4,\alpha=3,\beta=2,\gamma=1$, then
\[
0 = p_0(m)\mathfrak{f}_{m} + p_1(m) \mathfrak{f}_{m-1}
+ p_2(m) \mathfrak{f}_{m-2} + p_3(m) \mathfrak{f}_{m-3} + p_4(m) \mathfrak{f}_{m-4}
\]
with
\begin{gather*}
p_0(m) = 52 m^2 + 2674 m - 2726,
\\
p_1(m) = -4134 m^2 - 211907 m + 326301,
\\
p_2(m) = 109564m^2 + 5597900 m - 11612034,
\\
p_3(m) = -969462 m^2 - 49366597 m + 129474769,
\\
p_4(m) = 37180 m^2 + 1874730 m - 5958810. 
\end{gather*}
Moreover, since $p_0(m) \neq 0$ for all $m > 1$, the sequence $\{\mathfrak{f}_{m}\}$ can be defined using a $4$-term recurrence relation.
\end{example}

\begin{example} 
If $A=2, B=4, \alpha=1, \beta=1, \gamma=8/5$, then
\[
0 = p_0(m)\mathfrak{f}_{m} + p_1(m) \mathfrak{f}_{m-1}
+ p_2(m) \mathfrak{f}_{m-2} + p_3(m) \mathfrak{f}_{m-3} + p_4(m) \mathfrak{f}_{m-4}
\]
with
\begin{gather*}
p_0(m) = 1875 m^2 - 20625 m + 18750,
\\
p_1(m) = - 41000 m^2 + 457750 m - 573500,  
\\
p_2(m) = 303600 m^2 - 3443400 m + 5515600, 
\\ 
p_3(m) = - 796160 m^2 + 9191040 m - 17948160,
\\ 
p_4(m) = 258048 m^2 - 3096576 m + 6967296. 
\end{gather*}
In this case, $p_0(10)=0$ but $p_0(m)\neq0$ for all $m>10$. Thus, the sequence $\{\mathfrak{f}_{m}\}$ can be expressed in a recursive way only for $m>10$.
\end{example}

For some values of $A$ and $B$, the recurrence equation \eqref{eq:recgeneral} given in Proposition~\ref{prop:P-recursiva} is of order less than $4$, as seen in the following results.

\begin{proposition} 
If $A\in \{0,1\}$ or $B\in \{0,1\}$, then the recurrence relation defining $\{\mathfrak{f}_m\}$ is of the form
\begin{equation*}
\overline{p}_0(m) \mathfrak{f}_m + \overline{p}_1(m) \mathfrak{f}_{m-1} 
+ \overline{p}_2(m) \mathfrak{f}_{m-2} + \overline{p}_3(m) \mathfrak{f}_{m-3} = 0,
\end{equation*}
where the $\overline{p}_i(x)$ are polynomials of degree at most~$2$.
\end{proposition}

\begin{proof}
In \eqref{eq:recgeneral}, if $A=0$ or $B=0$, then $p_4(m)=0$. In the same way, if $A=1$ or $B=1$, then $p_0(m)=0$.
\end{proof}

\begin{proposition} 
If $A=B-1=0$, $A-1=B=0$, $A=B=0$, or $A=B=1$ then the recurrence relation defining $\{\mathfrak{f}_m\}$ is of the form
\begin{equation*}
\hat{p}_0(m) \mathfrak{f}_m + \hat{p}_1(m) \mathfrak{f}_{m-1} + \hat{p}_2(m) \mathfrak{f}_{m-2} = 0,
\end{equation*}
where the $\hat{p}_i(x)$ are polynomials of degree at most~$2$.
\end{proposition}

\begin{proof} 
In \eqref{eq:recgeneral}, if $A=B-1=0$ or $A-1=B=0$, then $p_4(m)=p_0(m)=0$. In the same way, if $A=B=0$, then $p_4(m)=p_3(m)=0$. Finally, if $A=B=1$, then $p_0(m)=p_1(m)=0$.
\end{proof}

\begin{remark}
Even though the polynomials $\overline{p}_i$ and $\hat{p}_i$ from the two preceding propositions are not the same polynomials from \eqref{eq:recgeneral}, they just result from a shift in the indexes. 
\end{remark}

We finish this section by analyzing the case $A=B$. At first sight, formula \eqref{eq:recgeneral} does not seem to be particularly simple if $A=B$. However, an independent proof can be used of find a 3-term recurrence relation in this case. The ultimate reason for this fact is that, if $A=B$, the second term in the generating function 
$G(z)$ in \eqref{eq:dossum} (see Theorem~\ref{th:GFfnn}) becomes simpler, so a much more direct method can be used. The details are given in the proof of the next proposition. Actually, we will assume that $A=B \ne 1$ because the case $A=B=1$ is already included in the previous proposition and it is well known \cite{HaKl1971} (recall that we are also assuming that $\alpha=\beta=1$).

 \begin{proposition} 
 If $A=B \ne 1$, the sequence $\{\mathfrak{f}_n\}$ is given, for every $m\geq 3$, by
\begin{align*}
  &(A-1) (m-1) \mathfrak{f}_{m} =
  \big( m (-4 + 4 A + A^2 - 2\gamma + 3 A\gamma)
  + 6 - 6 A - A^2 + 3\gamma - 4 A\gamma \big) \mathfrak{f}_{m-1} \\
  &\quad + \big( m (-4 A^2 - 4 A\gamma - 2 A^2\gamma + \gamma^2 - 3 A\gamma^2)
  + 6 A^2 + 6 A\gamma + 3 A^2\gamma - 2\gamma^2 + 5 A\gamma^2 \big) \mathfrak{f}_{m-2} \\
  &\quad + \big( m (A^2\gamma^2 + A\gamma^3) - 2A^2\gamma^2 - 2A\gamma^3 \big) \mathfrak{f}_{m-3},
  \end{align*}
and the first three terms are
\begin{equation*}
  \mathfrak{f}_{0} = 1, \qquad \mathfrak{f}_{1} = 2A+\gamma, \qquad
  \mathfrak{f}_{2} = 2 A^2 + 4 A (1+\gamma) + \gamma (2+\gamma).
\end{equation*}
\end{proposition}

\begin{proof} 
Let us write the generating function \eqref{eq:dossum} with $\alpha=\beta=1$ as
\[
  \sum_{m \ge 0} \mathfrak{f}_{m} z^{m} = P(z) + Q(z)
\]
with
\begin{align*}
  P(z) &= \frac{-B+1}{1 - B + B^2 z + \gamma B z} = \sum_{m \ge 0} p_{m} z^{m},
  \\
  Q(z) &= \frac{2 z \big(B + A - AB + \gamma\big) \big({-1} + \gamma z + S\big)}
  {S \big({-1} + 2 B z + \gamma z + S \big) \big(2 + A ({-1} + \gamma z + S)\big)}
  = \sum_{m \ge 0} q_{m} z^{m},
\end{align*}
and $S = \sqrt{1 + \gamma^2 z^2 - 2 (2+\gamma)z}$. 

With the previous notation we have $\mathfrak{f}_{m} = p_m + q_m$, so it is enough to find a way to evaluate $p_m$ and $q_m$. For $p_m$, let us note that
\begin{equation*}
  \frac{B-1}{B-1 - (B^2 + \gamma B) z}
  = \frac{1}{1 - \frac{B^2 + \gamma B}{B-1} z}
  = \sum_{m \ge 0} \bigg(\frac{B^2 + \gamma B}{B-1}\bigg)^m z^m,
\end{equation*}
so
\begin{equation}
\label{eq:forpm}
  p_m = \bigg(\frac{B (B + \gamma)}{B-1}\bigg)^m.
\end{equation}
Let us see how to compute $q_m$. The aim is to find a recurrence formula for $q_m$ using something similar to the method in~\cite{HaKl1971}, but now the corresponding expressions are much more complicated.

We must manipulate $Q(z)$. With the help of a computer algebra system, it is not difficult to check that the two factors of the denominator can be written as
\begin{equation*}
  \big({-1} + 2 B z + \gamma z + S \big) \big(2 + A ({-1} + \gamma z + S)\big) 
  = 2(1 - A + A B z + A \gamma z)
  \bigg( {-1} + \gamma z + \frac{2(B-A)z}{1 - A + A B z + A \gamma z} + S \bigg).
\end{equation*}
In this way, we have for $A=B$ that 
\[
  Q(z) = \frac{(2A - A^2 + \gamma)z}{(1 - A + A^2z + A\gamma z)S}.
\]
Differentiating, with the help of a computer algebra system again, we get
\[
  Q'(z) = \frac{(2A - A^2 + \gamma)(2 + \gamma - \gamma^2z)z}{(1 - A + A^2z + A\gamma z) S^3}
  + \frac{-(A-1)(2A-A^2+\gamma)}{(1 - A + A^2z + A\gamma z)^2 S}
  = Q_1(z) + Q_2(z).
\]
and it is easy to notice that
\begin{gather*}
  (1 + \gamma^2 z^2 - 2 z (2 + \gamma) ) Q_1(z) = (2 + \gamma - \gamma^2 z) Q(z),
  \\
  (1 - A + A^2 z + A \gamma z) z Q_2(z) = - (A-1) Q(z).
\end{gather*}
Using these formulas, and taking into account that $Q'(z) = Q_1(z)+Q_2(z)$, we obtain that $Q(z)$ satisfies the differential equation
\begin{align*}
&\bigg( (2 + \gamma - \gamma^2 z) (1 - A + A^2 z + A \gamma z) z
- (A-1) (1 + \gamma^2 z^2 - 2 z (2 + \gamma) ) \bigg) Q(z)
\\
&\qquad\qquad\qquad\qquad\qquad\qquad
= (1 + \gamma^2 z^2 - 2 z (2 + \gamma) ) (1 - A + A^2 z + A \gamma z) z Q'(z).
\end{align*}

If we substitute $Q(z) = \sum_{m \ge 0} q_m z^m$ and $z Q'(z) = \sum_{m \ge 0} m q_m z^m$ in the previous expression, and identify coefficients we obtain that
\begin{equation*}
  q_0 = 0, \quad q_1 = \frac{-2A + A^2 - \gamma}{A-1}, \quad
  q_2 = \frac{(-2A + A^2 - \gamma)(-2+2A+A^2+2A\gamma-\gamma)}{(A-1)^2}.
\end{equation*}
and that, for every $m\geq 3$,
\begin{equation}
\label{eq:recqm}
\begin{aligned}
  q_m &= \frac{1}{1-A + (A-1) m} \Big(
  \big( (m-1)A^2 + A ( (4+3\gamma) m - 6 - 4\gamma) - (2m-3) (2+\gamma) \big) q_{m-1} \\
  &\quad + \big( (m-2) \gamma^2 - A \gamma ((4 + 3\gamma)m - 6 - 5\gamma)
  - (2 m-3) A^2 (2 + \gamma) \big) q_{m-2} 
  + (m-2) A \gamma^2 (A + \gamma) q_{m-3}
  \Big),
\end{aligned}
\end{equation}

Finally, using \eqref{eq:forpm} and \eqref{eq:recqm}, it is not difficult to find the 
recursion for $\mathfrak{f}_m = p_m + q_m$, and the result follows.
\end{proof}

To close this section, let us finally note that, if we would be considering the general recurrence \eqref{eq:recu} (arbitrary $\alpha$ and $\beta$ with $\alpha\beta\neq 0$) instead of the recurrence \eqref{eq:recuab1} (in which $\alpha=\beta=1$), the hypothesis $A=B \ne 1$ of the previous proposition should be replaced by $A/\alpha = B/\beta \ne 1$ and the corresponding statement would be obtained just replacing $A\mapsto A/\alpha$, $B\mapsto B/\beta$, $\gamma\mapsto \frac{\gamma}{\alpha\beta}$ and $\mathfrak{f}_m\mapsto \frac{\mathfrak{f}_m}{\alpha^m\beta^m}$.


\section{Further work}
\label{sec:future}

\begin{figure}[p]
\begin{center}
\includegraphics[width=0.8\textwidth]{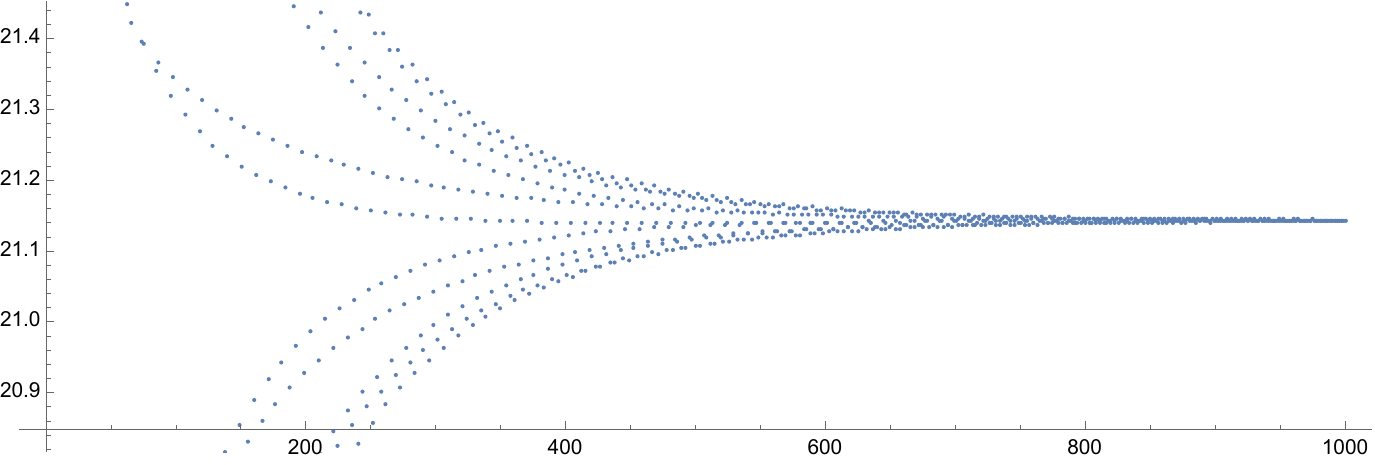}
\caption{$F_n\rightarrow 21.14\dots$ for $\{A,B,\alpha,\beta,\gamma\} = \{2, -4, -4, 3, 21\}$} 
\label{fig:limite}
\end{center}
\end{figure}

\begin{figure}[p]
\begin{center}
\includegraphics[width=0.8\textwidth]{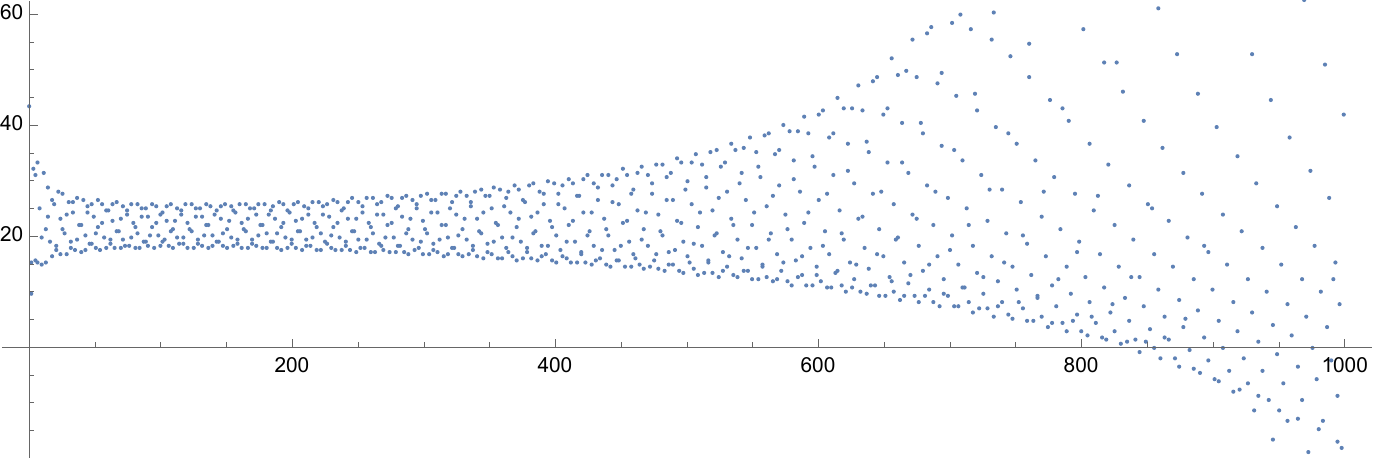}
\caption{$F_n$ is unbounded for $\{A,B,\alpha,\beta,\gamma\} = \{2, -4, -4, 3, 431/20\}$}
\label{fig:noacotado}
\end{center}
\end{figure}

\begin{figure}[p]
\begin{center}
\includegraphics[width=0.8\textwidth]{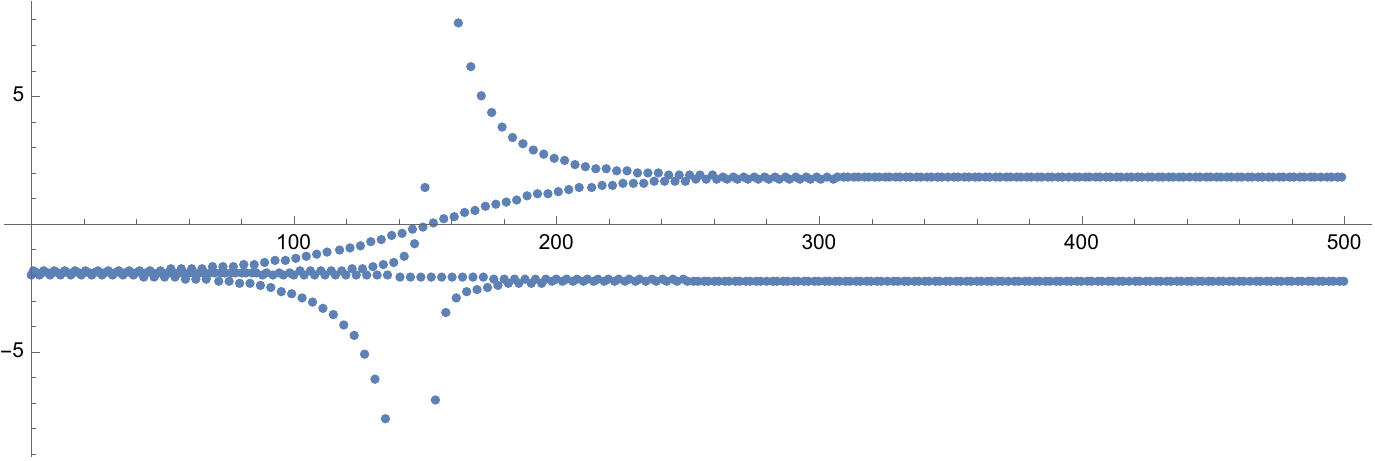}
\caption{$F_n$ tends to the $2$-cycle $\{1.81\dots,-2.20\dots\}$ for $\{A,B,\alpha,\beta,\gamma\} = \{27/20, -27/20, 1, 1, -2\}$}
\label{fig:doslimites}
\end{center}
\end{figure}

\begin{figure}[p]
\begin{center}
\includegraphics[width=0.8\textwidth]{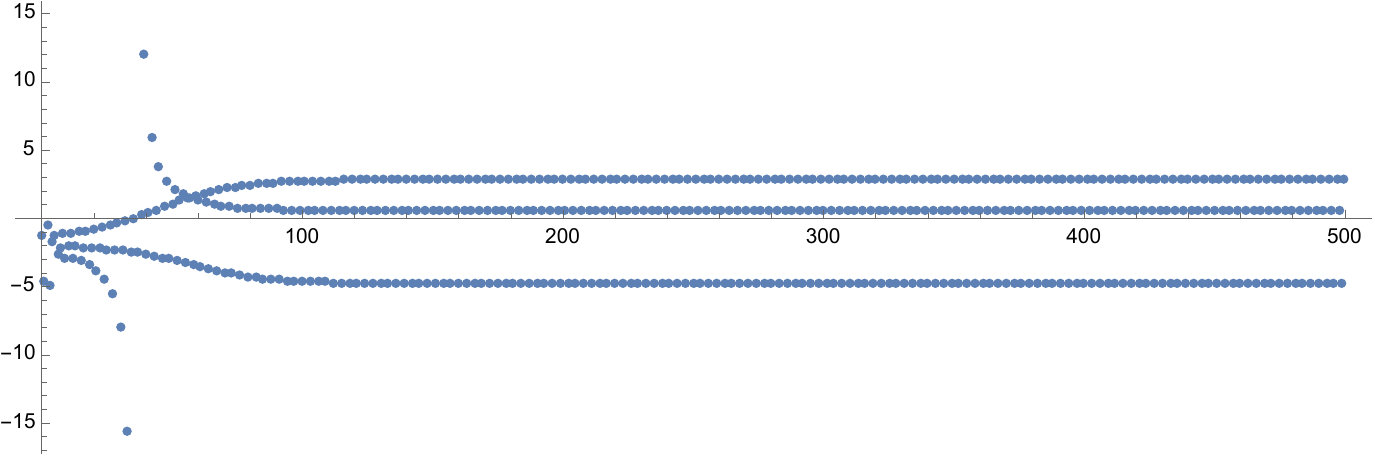}
\caption{$F_n$ tends to the $3$-cycle $\{2.83\dots,0.59\dots,-4.76\dots\}$ 
for $\{A,B,\alpha,\beta,\gamma\} = \left\{8/5, -8/5, 1, 3/2, -2\right\}$}
\label{fig:treslimites}
\end{center}
\end{figure}

In this work, given $A,B,\alpha, \beta, \gamma\geq 0$, we have analyzed the sequence defined by $f_{m,n}=\alpha f_{m-1,n}+\beta f_{m,n-1} + \gamma f_{m-1,n-1}$ with initial conditions $f(m,n)=A^m B^n $ for $mn=0$. We have paid special attention to the case $\alpha\beta\neq 0$ but our approach is equally valid otherwise. In this final section we expose some ideas to extend our work.

First of all, in might be interesting to allow for negative values of the parameters $A, B, \alpha, \beta$ and~$\gamma$. In this work, the sequence $F_n = \frac{f_{n+1,n+1}}{f_{n,n}}$ was always convergent. However, if no restrictions are considered, many different situations are possible. It can be convergent, it can be unbounded, it can be bounded having several limit points, etc. Figures~\ref{fig:limite}, \ref{fig:noacotado}, \ref{fig:doslimites} and~\ref{fig:treslimites} illustrate different possibilities.

It also seems interesting, and promising, to address the same sequence with more general initial conditions $f_{m,0}$ and $f_{0,n}$ that also admit reasonable interpretations in terms of paths, random walks, etc.~\cite{EdGr}. Thus, it might be worth trying to establish a connection between the sequences that we have addressed in this work and those appearing if we consider initial conditions $f_{m,0}= \mathcal{O}(A^m)$ and $f_{0,n}= \mathcal{O}(B^m)$. In fact, they seem to be totally related if the initial conditions are such that 
\begin{equation} 
\label{eq:limi} 
\limsup_m \sqrt[m]{f_{m,0}} = A \quad\text{ and }\quad \limsup_n \sqrt[n]{f_{0,n}} = B, 
\end{equation}
because the corresponding generating functions have the same radius of convergence. For example, if the initial conditions are given by the Fibonacci sequence, i.e., $f_{n,0} = f_{0,n} = \mathsf{Fib}(n)$ the limit $\frac{f_{n+1,n+1}}{f_{n,n}}$ exists and it has the same value as if the initial conditions were given by $f_{n,0}=f_{0,n}= \varphi^n$ with $\varphi=\frac{1}{2} \left(1+\sqrt{5}\right)$ (i.e., $A=B=\varphi$). Furthermore, the sequence is also P-recursive \cite[OEIS A344576]{A344576}.

Finally, we believe that it is worth considering the situation in which some of the above limits \eqref{eq:limi} do not exist, so that the radius of convergence of the generating function is zero. In these situations, the asymptotic behavior of $f_{n,n}$ is different, with the limit $\frac{f_{n+1,n+1}}{f_{n,n}}$ not existing in general. This happens, for instance, if $\alpha=\beta=\gamma=1$ and $f_{n,0}=f_{0,n}=n!$. In this example \cite[OEIS A346374]{A346374} we conjecture that $\frac{f_{n+1,n+1}}{f_{n,n}} \approx n+1$ (where $a_n \approx b_n$ means, as usual, that $a_n-b_n \to 0$ when $n \to \infty$). The same phenomenon happens if $f_{n,0}=f_{0,n}=n^n$, with our conjecture now being $\frac{f_{n+1,n+1}}{f_{n,n}} \approx e\cdot(n+1/2)$ \cite[OEIS A346385]{A346385}. However, in both cases the sequence $f_{n,n}$ is P-recursive.




\begin{thebibliography}{99}

\bibitem{BIO} 
H. Andrade, I. Area, J. J. Nieto, and \'A. Torres, 
The number of reduced alignments between two DNA sequences, 
\textit{BMC Bioinformatics} 
\textbf{15} (2014), article no.~94. 


\bibitem{Com1970} 
L. Comtet, 
\textit{Analyse combinatoire. Tomes I, II}, 
Presses Universitaires de France, Paris, 1970.

\bibitem{Dag20}
M. C. Da\u{g}l\i, 
A new generalization of Delannoy numbers, 
\textit{Indian J. Pure Appl. Math.} 
\textbf{51} (2020), 1729--1735.

\bibitem{Dag21}
M. C. Da\u{g}l\i, 
A new recursive formula arising from a determinantal expression for weighted Delannoy numbers, 
\textit{Turkish J. Math.} 
\textbf{45} (2021), 471--478.

\bibitem{Delannoy} 
H. Delannoy, 
Emploi de l'\'echiquier pour la resolution de certains probl\`emes de probabilit\'es,
\textit{Comptes-Rendus du Congr\`es annuel de l'Association Fran\c{c}aise pour l'Avancement des Sciences}, 
24, Bordeaux (1895), 70--90.

\bibitem{DeStWh}
J. Demeyer, W. Stein, and U. Whitcher,
Beyond the black box, 
\textit{Notices Amer. Math. Soc.} 
\textbf{63} (2016), 928--929. 

\bibitem{DonSh1977} 
R. Donaghey and L. W. Shapiro, 
Motzkin numbers, 
\textit{J. Combinatorial Theory Ser. A} 
\textbf{23} (1977), no.~3, 291--301.

\bibitem{DuPeVa} 
A. J. Dur\'an, M. P\'erez, and J. L. Varona, 
The misfortunes of a trio of mathematicians using computer algebra systems. Can we trust in them?, 
\textit{Notices Amer. Math. Soc.} 
\textbf{61} (2014), 1249--1252. 

\bibitem{EdGr}
S. Edwards and W. Griffiths, 
On generalized Delannoy numbers, 
\textit{J. Integer Seq.} 
\textbf{23} (2020), article 20.3.6.

\bibitem{FlSe2009} 
P. Flajolet and R. Sedgewick, 
\textit{Analytic combinatorics}, 
Cambridge University Press, Cambridge, 2009.

\bibitem{FrRo1971} 
R. D. Fray and D. P. Roselle,
Weighted lattice paths, 
\textit{Pacific J. Math.} 
\textbf{37} (1971), 85--96.

\bibitem{HaKl1971}
M. L. J. Hautus and D. A. Klarner,
The diagonal of a double power series,
\textit{Duke Math. J.}
\textbf{38} (1971), 229--235.

\bibitem{Mel2017}
S. Melczer,
\textit{Analytic Combinatorics in Several Variables: Effective Asymptotics and Lattice Path Enumeration},
Thesis, University of Waterloo, 2017.


\bibitem{Nara1955} 
T. V. Narayana, 
Sur les treillis form\'es par les partitions d'un entier et leurs
applications \`a la th\'eorie des probabilit\'es, 
\textit{C. R. Acad. Sci. Paris} 
\textbf{240} (1955), 1188--1189. 

\bibitem{No2012} 
R. Noble,
Asymptotics of the weighted Delannoy numbers,
\textit{Int. J. Number Theory} 
\textbf{8} (2012), no.~1, 175--188.

\bibitem{A001850}
{OEIS Foundation},
{Sequence A001850},
\textit{The On-Line Encyclopedia of Integer Sequences},
\url{http://oeis.org/A001850}

\bibitem{A008288}
{OEIS Foundation},
{Sequence A008288},
\textit{The On-Line Encyclopedia of Integer Sequences},
\url{http://oeis.org/A008288}

\bibitem{A344576}
{OEIS Foundation},
{Sequence A344576},
\textit{The On-Line Encyclopedia of Integer Sequences},
\url{http://oeis.org/A344576}

\bibitem{A346374}
{OEIS Foundation},
{Sequence A346374},
\textit{The On-Line Encyclopedia of Integer Sequences},
\url{http://oeis.org/A346374}

\bibitem{A346385}
\textsc{OEIS Foundation},
{Sequence A346385},
\textit{The On-Line Encyclopedia of Integer Sequences},
\url{http://oeis.org/A346385}

\bibitem{PeWi2002}
R. Pemantle and M. C. Wilson,
Asymptotics of multivariate sequences. I. Smooth points of the singular variety,
\textit{J. Combin. Theory Ser. A}
\textbf{97} (2002), no.~1, 129--161.

\bibitem{PeWi2004}
R. Pemantle and M. C. Wilson,
Asymptotics of multivariate sequences. II. Multiple points of the singular variety,
\textit{Combin. Probab. Comput.}
\textbf{13} (2004), no. 4-5, 735--761.

\bibitem{PeWi2008}
R. Pemantle and M. C. Wilson,
Twenty combinatorial examples of asymptotics derived from multivariate generating functions,
\textit{SIAM Rev.} 
\textbf{50} (2008), no.~2, 199--272.

\bibitem{PeWi2013}
R. Pemantle and M. C. Wilson,
\textit{Analytic combinatorics in several variables},
Cambridge Studies in Advanced Mathematics, 140,
Cambridge University Press, Cambridge, 2013.

\bibitem{QCSG2018}
F. Qi, V. \v{C}er\v{n}anov\'a, X.-T. Shi, and B.-N. Guo,
Some properties of central Delannoy numbers,
\textit{J. Comput. Appl. Math.}
\textbf{328} (2018), 101--115.


\bibitem{RaWi2007}
A. Raichev and M. C. Wilson,
A new method for computing asymptotics of diagonal coefficients of multivariate generating functions,
\textit{2007 Conference on Analysis of Algorithms, AofA 07}, 439--449,
\textit{Discrete Math. Theor. Comput. Sci. Proc., AH},
Assoc. Discrete Math. Theor. Comput. Sci., Nancy, 2007.
arXiv:math/0702595.

\bibitem{RaWi2008}
A. Raichev and M. C. Wilson,
Asymptotics of coefficients of multivariate generating functions: improvements for smooth points,
\textit{Electron. J. Combin.}
\textbf{15} (2008), no.~1, Research Paper 89, 17~pp.


\bibitem{SaHol1994} 
H. Sachs and H. Zernitz,
Remark on the dimer problem,
\textit{Discrete Appl. Math.} 
\textbf{51} (1994), no. 1-2, 171--179.

\bibitem{Sch1870} 
E. Schr\"oder,
Vier combinatorische Probleme,
\textit{Zeitschrift f\"ur Mathematik und Physik} 
\textbf{15} (1870), 361--376.

\bibitem{St1999} 
R. P. Stanley, 
\textit{Enumerative combinatorics}, Vol. 2, 
Cambridge Studies in Advanced Mathematics, 62,
Cambridge University Press, Cambridge, 1999.

\bibitem{Su2003} 
R. A. Sulanke,
Objects counted by the central Delannoy numbers,
\textit{J. Integer Seq.} 
\textbf{6} (2003), no.~1, Article 03.1.5, 19~pp.

\bibitem{BIO2} 
\'A. Torres, A. Cabada, and J. J. Nieto,
An exact formula for the number of alignments between two DNA sequences. 
\textit{DNA Sequence} 
\textbf{14} (2003), no.~6, 427--430.

\bibitem{Wan} 
Y. Wang, S.-N. Zheng, and X. Chen, 
Analytic aspects of Delannoy numbers, 
\textit{Discrete Math.} 
\textbf{342} (2019), 2270--2277.

\end{thebibliography}
\end{document}